\documentclass[10pt,reqno]{amsart}

\usepackage{amsthm,amsmath,amsfonts,amssymb}
\usepackage{bm}
\usepackage[dvipdfmx]{graphicx}
\usepackage{xcolor}
\usepackage{fullpage}

\usepackage{verbatim}
\usepackage{algorithm}
\usepackage{algpseudocode}

\usepackage{caption}
\usepackage{float}
\usepackage{subcaption}

\usepackage{tikz}
\usetikzlibrary{calc, positioning}
\usetikzlibrary{decorations.markings}
\usetikzlibrary{quotes,angles}
\usetikzlibrary{arrows.meta}
\usetikzlibrary{intersections}
\usetikzlibrary{graphs}

\usepackage{appendix}
\usepackage[sort&compress, numbers]{natbib}

\usepackage{hyperref}					
\hypersetup{colorlinks,					
	linkcolor=blue,
	citecolor=red}

\usepackage{amsaddr}
\usepackage{mathrsfs}

\newtheorem{theorem}{Theorem}

\newtheorem{lemma}{Lemma}

\theoremstyle{definition}
\newtheorem{definition}{Definition}
\newtheorem{remark}{Remark}

\def\R{\mathbb{R}}
\def\E{\mathbb{E}}

\def\N{\mathbb{N}}

\def\AA{\mathcal{A}}
\def\BB{\mathcal{B}}

\def\EE{\mathcal{E}}

\def\LL{\mathcal{L}}

\def\NN{\mathcal{N}}

\def\LD{\mathscr{D}}
\def\PLDR{\mathcal{M}} % Cartesian Product of Linear Differential Operator and Euclidean Space

\def\net{f}
\def\hyperbias{\zeta} % hyperparameter coefficient of bias term

\newcommand{\ol}{\overline}
\newcommand{\wt}{\widetilde}

\renewcommand{\d}{\mathrm{d}}

\def\llan{\left\langle}
\def\rran{\right\rangle}

\allowdisplaybreaks

\title{On the extreme eigenvalues of the Gram Matrix in Physics-Informed Neural Networks for the Poisson Equation}
\author{ Bangti Jin\and Longjun Wu} 
\address{ Department of Mathematics, The Chinese University of Hong Kong, Shatin, N.T., Hong Kong, P.R. China }

\begin{document}

\begin{abstract}
The smallest and largest eigenvalues of the Gram matrix induced by the differential neural tangent kernel (DNTK) play a pivotal role in the analysis of over-parameterized PINNs trained by gradient type algorithms. However, a theoretical analysis of the extreme eigenvalues remains completely absent due to the challenge posed by the presence of multiple differential operators.
In this work, we provide explicit lower and upper bounds for the extreme eigenvalues of the infinite DNTK matrix for the Poisson equation with the Dirichlet boundary condition for two-layer RePU neural networks without the bias term. The setting is fairly general with respect to the sampling points and input dimension \(d\): \(\delta\)-separated and additionally \(d\geq 3\) when deriving the lower bound of the smallest eigenvalue. These results extend that for the neural tangent kernel, and to the best of our knowledge, represent the first results on the spectrum of the DNTK.

\end{abstract}

\maketitle

\section{Introduction}

Physics-informed neural networks (PINNs) \cite{Raissi:2019PINN} (see also \cite{Lagaris1998,Sirignano:2018DGM} for related developments) are based on the principle of PDE residual minimization, and utilize neural networks (NNs) as the ansatz space for approximating the unknowns. Currently they represent one very popular class of neural PDE-solvers and have shown tremendous success in solving a diverse array of direct problems \cite{HuJinZhou:2024,JinCai:2021,Sirignano:2018DGM} and inverse problems \cite{JinLiQuan:2024,CenJinLi:2025}; see the review \cite{ToscanoKarniadakis:2025} and the references therein for further pointers. Due to these impressive empirical successes in many areas, there has been significant interest in developing relevant mathematical theory for PINNs \cite{DeRyckMishra:2024}. One notoriously challenging theoretical issue is the convergence analysis of various training algorithms, due to the strong nonconvexity of the (empirical) loss, arising from the nonlinearity of the NN output with respect to the NN parameters.
One powerful theoretical tool to analyze the training dynamics of over-parameterized NNs is  
the neural tangent kernel (NTK) theory \cite{Jacot:2018NTK}. This idea has recently been successfully generalized to PINNs, leading to the so-called
differential neural tangent kernel (DNTK) \cite{JinWu:2026DNTK}, 
to analyze the convergence of PINN training in the kernel regime. Specifically, in the regime of wide (two-layer) NNs, the training dynamics of PINNs trained by gradient type algorithms can be approximated by the kernel regression with the Gram matrix induced by the DNTK \cite{Chizat:2019Lazy} and then the positivity of the Gram matrix ensures that the training algorithm can find a global minimum of the loss in high probability.
So far, several global convergence results for wide two-layer PINNs have been established using the DNTK framework, e.g., gradient descent \cite{Gao:2023GD}, 
implicit gradient descent \cite{Xu:2024IGD} and natural gradient descent \cite{Xu:2024NGD} (all for the heat equation), and stochastic gradient descent and stochastic gradient flow for the Poisson equation \cite{JinWu:2025}.

In the optimization of wide PINNs, the spectrum of the DNTK Gram matrix plays a fundamental role, which also greatly affects the training efficiency. The positivity of the associated infinite Gram matrix ensures the convergence guarantee  \cite{GanLiLin:2025NTK}, while the smallest eigenvalue \cite{Gao:2023GD,Xu:2024IGD,JinWu:2025} and also largest eigenvalue \cite{Xu:2024NGD} determines the convergence rate of the training algorithms. Moreover, the spectrum, particularly extreme eigenvalues, are closely related to the spectral bias phenomenon during the PINN training by gradient type algorithms \cite{Chai2024:overcoming,Seroussi:2024,Faroughi:2025}, in which the algorithms predominantly capture low-frequency components while struggling to learn high-frequency ones. 
Recently Cheng et al \cite{Cheng:2026consistency} establish a necessary and sufficient condition for the consistency of the PINNs, characterized by the reproducing kernel Hilbert space of an integral operator induced by the DNTK. All the aforementioned discussions highlight the importance of a theoretical analysis of the spectrum of the DNTK in the mathematical analysis for PINNs. 

The spectral analysis of the DNTK is very involved due to the presence of multiple differential operators (i.e., PDE operator and initial / boundary conditions) in the DNTK, when compared that for the NTK. First, the structure of the DNTK-induced Gram matrix varies from one PDE to another, which considerably complicates a unified treatment of the Gram matrix of PINNs. Moreover, even for a fixed PDE, the entries of the Gram matrix take different forms depending on the involved operators, with the overall complexity depending heavily on the specific PDE at hand \cite{JinWu:2026DNTK}. 
Thus so far, the theoretical treatment of the DNTK spectrum remains quite scarce. Jin and Wu \cite{JinWu:2026DNTK} proved the positivity of the smallest eigenvalue of the DNTK matrices for smooth non‑polynomial and RePU activations, under the assumption that the sampling points are either pairwise distinct or pairwise non‑parallel depending on whether a bias term is present. Gan et al \cite{GanLiLin:2025NTK} studied the DNTK for the PINN loss with one single differential operator and proved that DNTK will not induce a faster eigenvalue decay rate or stronger spectral bias when compared with that of NTK. However, the setting in \cite{GanLiLin:2025NTK} greatly simplifies practical PINNs since the PDEs always involve at least two operators. This discrepancy highlights a key challenge: the analysis of the DNTK is substantially more complicated when multiple operators are present.
To the best of our knowledge, there is still no quantitative analysis of the spectrum of the DNTK-induced Gram matrix. 
Nonetheless, there are several empirical studies on  the spectrum of the DNTK matrices \cite{Wang:2022When,WangWang:2021eigen}, and several methods have been proposed to overcome the spectral bias for the training of PINNs \cite{Cai:2020phase,XuZhangCai:2025}. 

In this work, we derive explicit lower and upper bounds on the smallest and largest eigenvalues of the infinite DNTK matrix for two-layer bias-free neural networks with the RePU activation for solving the Poisson equation with the Dirichlet boundary condition; see Theorems \ref{T:bound-min} and \ref{T:bound-max} for the precise statements. In these results, the input dimension \(d\) and sampling points can be  arbitrarily given, subject only to the condition that the points are \(\delta\)-separated and additionally with \(d\geq 3\) in the lower bound of the smallest eigenvalue. These quantitative bounds shed insights into the role of the RePU activation order and the data distribution in shaping the DNTK spectrum. 
The key strategy to overcome the nonuniform expressions of the DNTK arising from distinct operators and the interior and boundary points is to transform the bounds for the eigenvalues into the infimum norm of certain functions defined on the unit sphere and then to use spherical harmonics to bound the derivatives of relevant quantities. The analysis techniques also apply to other symmetric differential operators, e.g., the Helmholtz and biharmonic equations, and Neumann boundary conditions on a sphere centered at the origin. To the best of our knowledge, these bounds represent the first quantitative results on the spectrum of the DNTK.

Now we briefly review the literature on eigenvalue estimates for NTK matrices, and compare the settings and theoretical findings. The work closest to ours is the work \cite{Karhadkar:2024Bounds}, in which, for the finite NTK matrix of deep ReLU networks, Karhadkar et al derived a lower bound for the smallest eigenvalues in terms of the distance between directions of sampling points and also upper bound in terms of the distance between sampling points, with \(d\geq3\) for lower bound. Despite the similarity of the characterizations of lower and upper bounds, there are several major differences in the assumption and analysis techniques. First, Karhadkar et al assume that the sampling points are from the unit sphere and then employ the hemisphere transform and spherical harmonics to derive a lower bound. In the DNTK case, there are different types of sampling points and differential operators, which does not naturally allow the unit sphere structure or apply the hemisphere transform. Instead, we exploit the symmetry of the differential operators and the homogeneity of the RePU activation to transform the estimate of the smallest eigenvalue to the norm of a spherical function. Second, we need to consider the minimum of sampling points and also the distance between sampling points for different types separately when estimating the largest eigenvalue in the DNTK case, since these situations are precisely the extreme cases that make the infinite DNTK matrix singular. Finally, the order of the lower bound for the smallest eigenvalue in Theorem \ref{T:bound-min} coincides with that in \cite{Karhadkar:2024Bounds} if \(s=0\) in Theorem \ref{T:bound-min}, but it requires \(s\geq 3\) in order to satisfy the regularity assumption due to the presence of second-order differential operators (i.e., the Laplacian). We also increase the order of the upper bound compared to that in \cite{Karhadkar:2024Bounds} since the infinite Gram matrix is different from the finite case. Moreover, our estimates are explicit without any implicit constant which gives a more explicit dependence of the bounds on the parameter setting. As to DNNs case, the bounds for the NTK matrices can be passed to deep neural networks based on the shallow case by using the recursive formula of the NTK. However, this strategy does not directly apply to the DNTK case since the recursive formula for DNTK is far more complicated and other techniques are needed for the eigenvalue estimate for the DNTK of DNNs.

There are also several works on the eigenvalue estimate for the NTK, although the settings are quite different from the present one. Nguyen et al. \cite{Nguyen:2021Tight} establish tight bounds on the smallest eigenvalue of the NTK matrices in probability sense for deep ReLU networks under both infinite-width limits and finite-width architectures with the data satisfying Lipschitz concentration property. However, these bounds are valid with high probability only if the input dimension grows logarithmically with the number of sampling points. Bombari et al. \cite{Bombari:2022} give a strictly positive lower bound on the minimum eigenvalue of NTK matrices for the DNNs with non-linear Lipschitz activations also under the Lipschitz concentration data distribution. The improvement also lies in the minimal over-parameterization such that the total number of neurons of the NN can be as little as \(\wt{\Omega}(\sqrt{N})\). However, the results also validated only in high-dimensional settings as they need \(N=o(d^2)\).
Montanari and Zhong \cite{MontanariZhong:2022} provide the lower bound of the minimum eigenvalue for the empirical NTK matrix for two-layer NN under the assumption that the training samples are uniformly drawn from the unit sphere and \(d=\wt{\Omega}(\sqrt[\ell]{N})\) for some integer \(\ell\). In contrast to the low-dimensional or fixed-input settings, their framework operates in an asymptotic regime where the input dimensionality \(d\) must grow at a rate with respect to the number $N$ of sampling points, which is also highly coupled with their architectural design constraints.
See also the relevant studies on the spectral analysis of NTK \citep{Bietti2019:inductive,Bietti:2021deep,Cao:2019towards} and its implications for optimization in practical scenarios \citep{CaoGu:2019,Zou:2020GD,ChenCao:2020}.

The rest of  this work is organized as follows. In Section \ref{S:DNTK}, we review the concept of (infinite) DNTK and describe the infinite Gram matrix associated with the PDEs. In Section \ref{S:Bound-eigen}, 
we give the main results, the bounds of extreme eigenvalues of infinite DNTK matrix. Throughout, the notation \([n]\) denotes the index set \(\{1,2,\ldots,n\}\) for any positive integer \(n\) and \(|\cdot|\) denotes the Euclidean norm of any vector \(x\in\R^d\). 

\section{The DNTK and Infinite Gram matrix}\label{S:DNTK}

In this section, we recall the DNTK framework \cite{JinWu:2026DNTK} for shallow neural networks (NNs), and describe the infinite Gram matrix induced by the DNTK in the context of PINNs for solving PDEs.

\subsection{Differential neural tangent kernel}
The DNTK \cite{JinWu:2026DNTK} generalizes the NTK \cite{Jacot:2018NTK} and integrates physical laws (in the form of differential operators) and NNs. Specifically, let \(\net(\cdot;\theta):\R^d\to\R\) be a fully connected NN with the parameter vector \(\theta\) initialized to independent and identically distributed (i.i.d.) standard Gaussian, and let \(\LD_d\) be the space of all linear differential operators acting on functions from \(\R^d\) to \(\R\). Then, in the DNTK framework, the output of the NN $f(\cdot;\theta)$ is a function on the Cartesian product space \(\PLDR_d:=\LD_d\times \R^d\):
\begin{equation}\label{E:DiffOut}
	\net:  \PLDR_d \to \R, \quad
	  (\AA, x) \mapsto \AA \net(x;\theta) 
	  := \AA_z \left. \net(z;\theta) \right|_{z=x}.
\end{equation}
Let \(x^\AA\) denote the point \((\AA, x)\in \PLDR_d\), called the \emph{\(\AA\)-marked point \(x\)}. 
The finite DNTK \(\Theta_{\net}\) of the NN \(\net(\cdot;\theta)\) is a kernel defined on the space \(\PLDR_d\) by
\begin{align*}
	\Theta_{\net}: \PLDR_d\times \PLDR_d \to \R,\quad  \left(x^\AA, y^\BB\right) \mapsto
	\llan \nabla_\theta \AA \net(x;\theta), \nabla_\theta \BB \net(y;\theta)  \rran.
\end{align*}
When the width of the NN tends to infinity and the activation $\sigma$ satisfies some mild 
regularity and growth assumptions, the NN output is a centered Gaussian process 
with the covariance function \(\Sigma_{\net}^\infty\) and the DNTK converges to a deterministic kernel 
\(\Theta_{\net}^\infty\), both on the space \(\PLDR_d\); see \cite{JinWu:2026DNTK} for more details. 
These two kernels are positive in the following sense: 
given any set of finite points \(X=\{x_1^{\AA_1}, \ldots,x_N^{\AA_N}\}\subset\PLDR_d\), the kernel matrices 
\[  \Sigma_{\net}^\infty(X) = \left[\Sigma_{\net}^\infty(x_i^{\AA_i},x_j^{\AA_j})\right]_{i,j\in[N]}
		\quad \text{ and } \quad 
	\Theta_{\net}^\infty(X) = \left[\Theta_{\net}^\infty(x_i^{\AA_i},x_j^{\AA_j})\right]_{i,j\in[N]} \]
are positive semidefinite. 
Moreover, if \(\net(\cdot;\theta)\) is a shallow neural network of the form
\begin{equation}\label{E:2-layerNN}
	f^{(2)}(x;\theta) = d_1^{-\frac12} W^{(2)} 
	\sigma ( W^{(1)}x + \gamma b^{(1)}), \quad x\in\R^d,
\end{equation}
with \( W^{(2)} \in \mathbb{R}^{1 \times d_1}, W^{(1)} \in \mathbb{R}^{d_1 \times d}, 
b^{(1)} \in \mathbb{R}^{d_1}\) and \(\gamma\) being the weight of the bias term, then the infinite covariance function \(\Sigma_{\net}^\infty\) is explicitly given by 
\begin{equation}\label{E:2-Cov}		\Sigma^{(2)}_{\infty}\left(x^\AA,y^\BB\right) = \E_{u\sim \mathcal{N}(0,I_d), v\sim \mathcal{N}(0,1)}
		\left[ \mathcal{A}\sigma\left(u^\top x +  \gamma v\right)\mathcal{B}\sigma\left(u^\top y +  \gamma v\right) \right] 
	 	+ \gamma^2 a(x)b(y),
	\end{equation} 
where $a$ and $b$ denote the zeroth-order terms of the operators $\mathcal{A}$ and $\mathcal{B}$, respectively. Meanwhile, the infinite DNTK \(\Theta_{\net}^\infty\) is given by
\begin{align}
	\Theta^{(2)}_{\infty}\left( x^\AA,y^\BB \right) 
		=& \E_{u\sim \mathcal{N}(0,I_d), v\sim \mathcal{N}(0,1)}
		\left[ \mathcal{A}\sigma\left(u^\top x + \gamma v\right)
		\mathcal{B}\sigma\left(u^\top y + \gamma v\right) \right] + \gamma^2 a(x)b(y) \nonumber\\
		&  + \E_{u\sim \mathcal{N}(0,I_d), v\sim \mathcal{N}(0,1)}
		\left[ \sum_{k=1}^{d}\mathcal{A}\left(\sigma'\left(u^\top x + \gamma v\right)x_k\right)
		\mathcal{B}\left(\sigma'\left(u^\top y + \gamma v\right)y_k\right)  \right] \nonumber\\
		&  + \gamma^2 \E_{u\sim \mathcal{N}(0,I_d), v\sim \mathcal{N}(0,1)}
		\left[ \mathcal{A}\sigma'\left(u^\top x + \gamma v\right)
		\mathcal{B}\sigma'\left(u^\top y + \gamma v\right)  \right].\label{E:2-DNTK}
	\end{align}
where $x_k$ and $y_k$ denote the \(k\)th component of the input points $x\in \Omega$ and $y\in\partial\Omega$, respectively. 

\subsection{Gram matrix induced by the DNTK}
The Gram matrix induced by the DNTK governs the training dynamics of PINNs. We illustrate the idea with the boundary value problems for second-order elliptic differential equations. Let $\Omega\subset \mathbb{R}^d$ be an open bounded domain with a smooth boundary $\partial\Omega$. 
Consider the following boundary value problem:
\begin{equation}\label{E:PDE}
	\left\{
	\begin{aligned}
		\LL u &=\phi, \quad \mbox{in } \Omega, \\
		\BB u&=\psi, \quad  \mbox{on }  \partial \Omega,
	\end{aligned}
	\right.
\end{equation}
where \(\LL\) and $\BB$ represent the PDE operator and the boundary operator (e.g., the identity operator for the Dirichlet boundary condition), respectively.
To approximate the PDE solution $u$ by the NN \(\net(\cdot;\theta):\R^d\to\R\), the PINN \cite{Raissi:2019PINN} minimizes the following empirical loss:
\begin{equation}\label{E:Emp_Loss} % empirical loss
	L(\theta) := \sum_{p=1}^{N_1} \frac{1}{2N_1}\left( \LL \net(x_p;\theta)-\phi(x_p)\right)^2
		 + \sum_{q=1}^{N_2} \frac{\hyperbias^2}{2N_2}\left( \BB \net(y_q;\theta)-\psi(y_p)\right)^2,
\end{equation}
where \(X = \{x_p\}_{p=1}^{N_1}\) and \(Y = \{{y}_q\}_{q=1}^{N_2}\) are the sampling points drawn 
from the domain $\Omega$ and the boundary $\partial\Omega$, respectively, and the hyperparameter \(\hyperbias>0\) balances the PDE residual loss and boundary loss, denoted by \(s_p(\theta)\) and \(h_q(\theta)\), respectively: 
\begin{gather*}
	s_p(\theta) = \alpha\left(\LL \net(x_p;\theta)-\phi(x_p)\right) \quad \text{ and } \quad
	h_q(\theta) = \beta\left(\BB \net(y_q;\theta)-\psi(y_p)\right),
\end{gather*}
with \(\alpha:=N_1^{-\frac12}\) and \(\beta := \hyperbias N_2^{-\frac12}\). 
Let \(s(\theta)\) and \(h(\theta)\) be the interior and boundary loss vectors:
\begin{gather*}
	s(\theta)=[s_1(\theta), \cdots, s_{N_1}(\theta)]^\top\in \mathbb{R}^{N_1} \quad	\text{ and } \quad
	h(\theta)=[h_1(\theta), \cdots, h_{N_2}(\theta)]^\top \in \mathbb{R}^{N_2}.
\end{gather*}

In practice, the PINN loss $L(\theta)$ is often minimized using gradient type algorithms, which is a suitable discretization of the gradient flow, which updates the parameter vector \(\theta\) as 
\begin{equation}\label{E:GF-emp} 
	\frac{\d\theta_t}{\d t} = - \nabla_\theta L(\theta_t)
	= - \sum_{p=1}^{N_1}  s_p(\theta_t) \nabla_\theta s_p(\theta_t) 
		- \sum_{q=1}^{N_2}  h_q(\theta_t) \nabla_\theta h_q(\theta_t).
\end{equation}
Then the dynamics of the loss vector is governed by the ODE system \cite{JinWu:2026DNTK}:
\begin{equation}\label{E:dyn-emp-G} 
	\frac{\d}{\d t}\begin{bmatrix} s(\theta_t) \\ h(\theta_t) \end{bmatrix} 
	= - G(\theta_t)\begin{bmatrix} s(\theta_t) \\ h(\theta_t) \end{bmatrix}, 
\end{equation}
with the Gram matrix $G(\theta_t)$ given by \(G(\theta_t) = D(\theta_t)^\top D(\theta_t) 
\in \mathbb{R}^{(N_1 + N_2)\times (N_1 + N_2)} \) such that
\begin{align*}
	D(\theta_t) & = \begin{bmatrix}
		 \nabla_\theta s_1(\theta_t)
		& \cdots 
		&  \nabla_\theta s_{N_1}(\theta_t) 
		&  \nabla_\theta h_1(\theta_t)
		& \cdots 
		&  \nabla_\theta h_{N_2}(\theta_t)
		\end{bmatrix}\\
        &= \begin{bmatrix} \nabla_\theta \alpha\LL \net(X;\theta_t) &
	 	\nabla_\theta \beta\BB \net(Y;\theta_t) \end{bmatrix}.
\end{align*}
Note that the Gram matrix \(G(\theta_t)\) is completely characterized by the DNTK: 
\begin{equation}\label{E:Gram-DNTK}
    G(\theta_t) = \Theta_{\net(\cdot;\theta_t)}\left( \{X^{\alpha\LL},Y^{\beta\BB} \} \right) =         \begin{bmatrix}
		\Theta_{\net(\cdot;\theta_t)}(X^{\alpha\LL},X^{\alpha\LL}) & \Theta_{\net(\cdot;\theta_t)}(X^{\alpha\LL},Y^{\beta\BB}) \\
		\Theta_{\net(\cdot;\theta_t)}(Y^{\beta\BB},X^{\alpha\LL}) & \Theta_{\net(\cdot;\theta_t)}(Y^{\beta\BB},Y^{\beta\BB})
	\end{bmatrix}, 
\end{equation}
where \(\Theta_{\net}(X^{\AA},Y^{\BB})\) denotes the matrix \([\Theta_{\net}(x_i^{\AA},y_j^{\BB})]_{i\in[N_1],j\in[N_2]}\in \R^{N_1\times N_2}\), corresponding to the two finite sets of sampling points \(X^{\AA}=\{x_1^{\AA},\ldots,x_{N_1}^{\AA}\}\) and \(Y^{\BB}=\{y_1^{\BB},\ldots,y_{N_1}^{\BB}\}\) that are marked by \(\AA\) and \(\BB\), respectively. 
When the hidden layers of the NN \(\net(\cdot;\theta)\) are wide enough, the initial Gram matrix \(G(\theta_0)\) approximates the infinite Gram matrix \(G^\infty=\Theta_{\net}^\infty\left( \{X^{\alpha\LL},Y^{\beta\BB} \} \right)\) due to the concentration property of Gaussian random variables, and then stays nearly constant in the lazy training regime. % due to the continuity of Frobenius norm of \(G(\theta_t)\). 

The spectral property of $G^\infty$ determines the dynamics of the ODE \eqref{E:GF-emp}. Let \(\{q_i\}_{i=1}^{N_1+N_2}\) be the orthogonal eigenvectors of the positive semidefinite matrix \(G^\infty\) with the  eigenvalues \(\{\lambda_i\}_{i=1}^{N_1+N_2}\). Then, using the orthogonal decomposition of  \(G^\infty\), the loss vector in \eqref{E:dyn-emp-G} for the infinitely wide PINNs evolves as follows: At the initialization it has a random norm in the direction \(q_i\) for any \(i\in[N_1+N_2]\), and then decays exponentially with the rate \(e^{-\lambda_i t}\) along the direction \(q_i\) (or along the \(\lambda_i\)-eigenspace if the multiplicity of \(\lambda_i\) is larger than one). Thus, the loss vector along the directions associated with larger eigenvalues decay faster than that with smaller eigenvalues, from which the spectral bias arises.

\section{Bounds on the extreme eigenvalues}\label{S:Bound-eigen}

In this section, we state lower and upper bounds on the extreme eigenvalues 
of the infinite DNTK matrix $\Theta^\infty_f$ for the Poisson equation with the Dirichlet boundary condition, for which
the operators $\LL$ and $\BB$ in \eqref{E:PDE} are given by 
\begin{equation}\label{E:Poisson}
\LL = \Delta \quad\mbox{and} \quad \BB = \EE,
\end{equation}
with $\EE$ being the identity operator. Consider two-layer NNs with the rectified power unit (RePU) activation function without the bias term (i.e., the NN \eqref{E:2-layerNN} with \(\gamma=0\)):
\begin{equation}\label{E:2-NN_No-bias}
	f(x;\theta) = d_1^{-1/2} W^{(2)} \sigma ( W^{(1)}x), \quad x\in\R^d,
\end{equation}
with the RePU activation \(\sigma(t) = \sigma_s(t) := (t)_+^s \),  where \((t)_+ = t\) for \(t\geq0\) and zero otherwise.
The covariance function \(\Sigma^{\infty}_{\net}\) \eqref{E:2-Cov} for the infinite-width NN \eqref{E:2-NN_No-bias} is given by 
\begin{equation}\label{E:2-Cov_no-bias}
	\Sigma^{\infty}_{\net}\left(x^\AA,y^\BB\right) = \E_{u\sim \mathcal{N}(0,I)}
	\left[ \AA\sigma(u^\top x)\BB\sigma(u^\top y) \right],
\end{equation}
while the infinite DNTK \(\Theta^{\infty}_{\net}\) is given by
\begin{align}
\Theta^{\infty}_{\net}\left( x^\AA,y^\BB \right) 
	& = \E_{u\sim \mathcal{N}(0,I)}
	\left[ \AA\sigma\left(u^\top x \right)\BB\sigma\left(u^\top y \right) +
	\sum_{k=1}^{d}\AA\left(\sigma'\left(u^\top x \right)x_k\right)
	\BB\left(\sigma'\left(u^\top y \right)y_k\right)  \right] \nonumber\\
	& = \Sigma^{\infty}_{\net}\left( x^\AA,y^\BB \right) + \sum_{k=1}^{d}\E_{u\sim \mathcal{N}(0,I)}
	\left[\AA\left(\sigma'\left(u^\top x \right)x_k\right)
	\BB\left(\sigma'\left(u^\top y \right)y_k\right)  \right].\label{E:2-DNTK_no-bias}
\end{align}
Below we assume \(s\geq 3\) so that the PINN loss is well defined (i.e., the PDE \eqref{E:PDE} holds in the strong sense) and also the gradient  of the loss with respect to $\theta$ is well defined. Also we study only the case \(d\geq 2\) (for $d=1$, $\partial\Omega$ contains only two points), and for the lower bound of the smallest eigenvalue, we require additionally \(d\geq 3\); see Remark \ref{R:bound-min} for further discussions.

Given the sampling points \(X_{I} = \{x_p\}_{p=1}^{N_1}\subset \Omega\) and \(X_{B} = \{{y}_q\}_{q=1}^{N_2}\subset\partial\Omega\), let \(N:=N_1+N_2\), 
\(X_{B} = \{x_{N_1+1},\ldots, x_{N_1+N_2}\}\), and   
\(X = X_{I} \sqcup X_{B} = \{x_\ell\}_{\ell=1}^{N} = \{x_1,\ldots, x_{N_1}, x_{N_1+1},\ldots, x_{N_1+N_2}\} \). 
Then for the Poisson equation \eqref{E:Poisson}, the covariance matrix \(\Sigma^\infty_{\net}(X)\in \R^{N\times N}\) is given by 
\begin{equation}\label{E:Sigma(X)}
	\Sigma^\infty_{\net}(X) = \begin{bmatrix}
	\Sigma^\infty_{\net}(X_{I},X_{I}) & \Sigma^\infty_{\net}(X_{I},X_{B}) \\
	\Sigma^\infty_{\net}(X_{B},X_{I}) & \Sigma^\infty_{\net}(X_{B},X_{B})
\end{bmatrix},
\end{equation}
where the submatrices \(\Sigma^\infty_{\net}(X_{I},X_{I})\in \R^{N_1\times N_1}\), \(
\Sigma^\infty_{\net}(X_{B},X_{B})\in \R^{N_2\times N_2}\) and 
\( \Sigma^\infty_{\net}(X_{I},X_{B}) = \Sigma^\infty_{\net}(X_{B},X_{I})^\top \in \R^{N_1\times N_2} \)
with the entries given by (cf. \eqref{E:2-Cov_no-bias})
\begin{align*}
	\Sigma^\infty_{\net}(x_i,x_j) & = \E_{u\sim \mathcal{N}(0,I)}
		\left[ \Delta\sigma(u^\top x_i)\Delta\sigma(u^\top x_j) \right],\quad \text{ for } i,j\in[N_1],\\ 
	\Sigma^\infty_{\net}(y_i,y_j) & = \E_{u\sim \mathcal{N}(0,I)}
		\left[ \sigma(u^\top y_i)\sigma(u^\top y_j) \right],\quad \text{ for } i,j\in[N_2],\\
	\Sigma^\infty_{\net}(x_p,y_q) & = \E_{u\sim \mathcal{N}(0,I)}
	\left[ \Delta\sigma(u^\top x_p)\sigma(u^\top y_q) \right],\quad \text{ for } p\in[N_1],q\in[N_2]. 
\end{align*}
The infinite DNTK matrix \(\Theta^\infty_{\net}(X)\) has a similar form
\begin{equation}\label{E:Theta(X)}
	\Theta^\infty_{\net}(X) = \begin{bmatrix}
	\Theta^\infty_{\net}(X_{I},X_{I}) & \Theta^\infty_{\net}(X_{I},X_{B}) \\
	\Theta^\infty_{\net}(X_{B},X_{I}) & \Theta^\infty_{\net}(X_{B},X_{B})
\end{bmatrix},
\end{equation}
where the entries are given by the formula \eqref{E:2-DNTK_no-bias} with the Laplace operator \(\Delta\) for the 
interior points \(X_I\) while the identity operator \(\EE\) for the boundary points \(X_B\).

\begin{remark}
The infinite DNTK matrix \(\Theta^\infty_{\net}(X) = \Theta^\infty_{\net}(\{X_I^{\Delta}, X_B^{\EE}\})\) in \eqref{E:Theta(X)} is slightly different from the infinite Gram matrix \(G^\infty(X)\) associated to the PINN loss \eqref{E:Emp_Loss}. 
Indeed, from the expression \eqref{E:Gram-DNTK}, \(G^\infty(X)\) should be \(\Theta^\infty_{\net}(\{X_I^{\alpha\Delta}, X_B^{\beta\EE}\})\) with 
	the constants \(\alpha\) and \(\beta\) arising from the empirical loss \eqref{E:Emp_Loss}. The Gram matrix \(G^\infty(X)\) is a diagonal scaling of DNTK matrix \(\Theta^\infty_{\net}(X)\). Below we discard the constants \(\alpha\) and \(\beta\). %, which also depends on the hyperparameter \(\hyperbias\) in \eqref{E:Emp_Loss}. 
\end{remark}

Now, we analyze the extreme eigenvalues of the matrices $\Sigma_f^\infty$ in \eqref{E:Sigma(X)} and $\Theta_f^\infty$ in \eqref{E:Theta(X)}. Below we sketch the derivation of a lower bound of the smallest eigenvalue  to illustrate the proof techniques used to cope with the nonuniform expression of entries in the DNTK matrix due to the presence of multiple operators for the eigenvalue estimates. The key tool in the analysis is spherical harmonics. The detailed proofs are in Section \ref{S-A:Bound-eigen}.
First, to ensure the positive definiteness of the matrices \(\Sigma^\infty_{\net}(X)\) and \(\Theta^\infty_{\net}(X)\), we need mild assumptions on \(X\) (due to the absence of the bias term). Suppose that \(X\) are nonzero and pairwise non-parallel \cite{JinWu:2026DNTK}. To quantify the latter property, we use the concept of \(\delta\)-separateness \cite{Karhadkar:2024Bounds,Allen:2019}. 

\begin{definition} A finite set of nonzero points \(X\) is said to be \emph{\(\delta\)-separated} for some \(\delta>0\) if 
\begin{equation}\label{E:delta-sepa}
	1 - \langle\ol{x}, \ol{y}\rangle^2 \geq \delta^2, \quad  \text{ for any } x\neq y\in X, \text{ with } 
	\ol{x} := x/|x|, ~ y:=y/|y|.
\end{equation}
\end{definition}

Let \(w = (w_1,\ldots,w_{N_1},w_{N_1+1},\ldots,w_N) =: (w^{I},w^{B}) 
\in \mathbb{S}^{N-1} \subset \R^N = \R^{N_1+N_2}\) be a test vector on $\mathbb{S}^{N-1}$. The expression \eqref{E:Sigma(X)} of the matrix \(\Sigma^{\infty}_{\net}(X)\)  (and the kernel expression \(\Sigma^{\infty}_{\net}\) in \eqref{E:2-Cov_no-bias}) motivates the induced norm:
\begin{equation}\label{E:Inner-Cov} % inner product by covariance matrix
	w^\top \Sigma^{\infty}_{\net}(X) w 
	= \E_{u\sim \NN(0,I_d)} \bigg[ \Big| \sum_{p=1}^{N_1} w^{I}_p \Delta\sigma(u^\top x_p) + 
	\sum_{q=1}^{N_2} w^{B}_q \sigma(u^\top y_q) \Big|^2 \bigg].
\end{equation}
Similarly, the norm induced by the DNTK matrix \(\Theta^{\infty}_{\net}(X)\) is given by
\begin{equation} \label{E:Inner-DNTK} % inner product by DNTK matrix
	w^\top \Theta^{\infty}_{\net}(X) w  = w^\top \Sigma^{\infty}_{\net}(X) w + \sum_{k=1}^{d}
	\E_{u\sim \NN(0,I_d)} \left[ \left| g_k(u;w) \right|^2 \right], 
\end{equation}
with the functions \(g_k(u;w)\) given by
\begin{equation} \label{E:g_k(u;w)}
	g_k(u;w) = \sum_{p=1}^{N_1} w^{I}_p \Delta\left(\sigma'\left(u^\top x_p \right)x_{pk}\right) + 
	\sum_{q=1}^{N_2} w^{B}_q \left(\sigma'\left(u^\top y_q \right)y_{qk}\right).
\end{equation}
Clearly, the second term in \eqref{E:Inner-DNTK} is always nonnegative. 
By Weyl min-max principle, every lower bound of the smallest eigenvalue for \(\Sigma^{\infty}_{\net}(X)\) 
is also a lower bound for that of \(\Theta^{\infty}_{\net}(X)\): 
\begin{equation}\label{E:lambda-The_Sig} % lambda(Theta), lambda(Sigma)
	\lambda_{\min}(\Theta^{\infty}_{\net}(X)) = 
	\inf_{w\in \mathbb{S}^{N-1}} w^\top \Theta^{\infty}_{\net}(X) w \geq
	\inf_{w\in \mathbb{S}^{N-1}} w^\top \Sigma^{\infty}_{\net}(X) w 
	= \lambda_{\min}(\Sigma^{\infty}_{\net}(X)).
\end{equation}
Since the Gaussian vector \(u\sim \NN(0,I_d)\) is the product of two independent
random variables \( u = |u|\cdot \ol{u} \), where the squared Euclidean norm \(|u|^2 \sim \chi^2_d \) follows the 
\(\chi^2_d\)-distribution and the direction \( \ol{u} = u/|u| \sim \mathrm{U}(\mathbb{S}^{d-1})\) 
follows the uniformly distribution over $\mathbb{S}^{d-1}$. 
By substituting \(\Delta\sigma(u^\top x_p) = \sigma^{(2)}(u^\top x_p)|u|^2\) and \(\sigma=\sigma_s\) 
into \eqref{E:Inner-Cov}, we arrive at
\begin{align}\label{E:Pro-trans} 	w^\top \Sigma^{\infty}_{\net}(X) w 
	& = \E_{u\sim \NN(0,I_d)} \left[ \left| \sum_{p=1}^{N_1} w^{I}_p \sigma_s^{(2)}(u^\top x_p)|u|^2 + 
	\sum_{q=1}^{N_2} w^{B}_q \sigma_s(u^\top y_q) \right|^2 \right]\nonumber \\
	& = \E_{u\sim \NN(0,I_d)} \left[ \left| \sum_{p=1}^{N_1} w^{I}_p 
	\sigma_s^{(2)}\Big(\ol{u}^\top x_p\Big)|u|^s + 
	\sum_{q=1}^{N_2} w^{B}_q \sigma_s\Big(\ol{u}^\top y_q\Big)|u|^s  \right|^2\right] \nonumber\\
	& = \E_{u\sim \mathrm{U}(\mathbb{S}^{d-1})} \left[ \left| \sum_{p=1}^{N_1} w^{I}_p 
	\sigma_s^{(2)}(u^\top x_p) + \sum_{q=1}^{N_2} w^{B}_q \sigma_s(u^\top y_q)  \right|^2\right]
	\cdot \E_{|u|^2\sim \chi^2_d} \left[|u|^{2s}\right] \nonumber\\
	& = \E_{|u|^2\sim \chi^2_d} \left[|u|^{2s}\right] \E_{u\sim \mathrm{U}(\mathbb{S}^{d-1})}\left[|g(u;w)|^2\right]. 
\end{align}	
The function \(g(u;w):\mathbb{S}^{d-1} \to \R \) is given by
\begin{equation}\label{E:g(u;w)}
\begin{aligned}
	g(u;w) & = \sum_{p=1}^{N_1} w^{I}_p 
	\sigma_s^{(2)}(u^\top x_p) + \sum_{q=1}^{N_2} w^{B}_q \sigma_s(u^\top y_q) \\
	& = \sum_{p=1}^{N_1} w^{I}_p (s)_2|x_p|^{s-2}\sigma_{s-2}(u^\top \ol{x}_p) 
	+ \sum_{q=1}^{N_2} w^{B}_q |y_q|^{s}\sigma_s(u^\top \ol{y}_q),	
\end{aligned}
\end{equation}
where \((s)_k = s(s-1)\cdots(s-k+1)\), for \(k\geq 0\), denotes the falling factorial 
and \(\ol{x}:= x/|x|\) denotes the normalized vector for any nonzero \(x\in\R^d\).
Thus, the bounds of the smallest eigenvalue \(\lambda_{\min}(\Sigma^{\infty}_{\net}(X))\) are related to the average norm of the function \(g(u;w)\) on the unit sphere for all \(w\in \mathbb{S}^{d-1}\), which can be analyzed using tools from spherical harmonics.

\begin{remark}
The reduction of the problem from the lower bound of the eigenvalue to that of function norms on  $\mathbb{S}^{d-1}$ requires several properties of the differential operator and the activation $\sigma$. First, the operator \(\AA\) is symmetric, including identity, Laplacian, Helmholtz and biharmonic operators etc, or the outer derivative if the boundary $\partial\Omega$ is a sphere centered at the origin, so that the expression of \(\AA\sigma(u^\top x)\) is tractable, like \(\Delta\sigma(u^\top x_p) = \sigma^{(2)}(u^\top x_p)|u|^2\). Then we can exploit the polar decomposition of a Gaussian vector. 
Second, the absence of the bias term and the homogeneity of the activation \(\sigma\) allow  extracting a \(\chi^2_d\) random variable, i.e., the power of the norm of the standard Gaussian vector, from the inner product \(w^\top \Sigma^{\infty}_{\net}(X) w\) as in \eqref{E:Pro-trans}.
Thus the analysis technique does not directly apply to other activation functions, non-symmetric differential operators or two-layer NNs with bias term. Also it remains very challenging to analyze the DNTK for DNNs since the recursive formulas of the covariance function and DNTK are complicated unlike the NTK case \cite{JinWu:2026DNTK}.
This also reflects the technical challenge for the spectral analysis of the DNTK.
\end{remark}

Now, we can state the main results of the work. Given the sets of interior and boundary points \(X=X_{I} \sqcup X_{B}\), we define the following two constants:
\begin{equation}\label{E:pts-LUbnd} % sampling points, lower and upper bound
	b_s(X):= \min_{x_p\in X_I, ~ y_q\in X_B}\left\{ (s)_2|x_p|^{s-2}, |y_q|^{s} \right\}\quad\mbox{and} \quad 
	B_s(X):= \max_{x_p\in X_I, ~ y_q\in X_B}\left\{ (s)_2|x_p|^{s-2}, |y_q|^{s} \right\}.
\end{equation}	

\begin{theorem}[Bounds on the smallest eigenvalue] \label{T:bound-min}
	Let \(\Theta^\infty_{\net}(X)\) be the infinite DNTK matrix for the 
	Poisson equation \eqref{E:Poisson} with the sampling points \(X\). 
\begin{itemize}
\item[{\rm(i)}] If $X$ is \(\delta\)-separated, then for \(d\geq 3\) and \( 3\leq s\leq n_1-5\), \(\lambda_{\min}(\Theta^\infty_{\net}(X))\) is lower bounded: 
		\begin{equation*}% \label{E:min-lower}
			\lambda_{\min}(\Theta^\infty_{\net}(X))
			\geq \frac{2^{s+d+2}}{e\pi^2} \frac{\Gamma(s-1)^2\Gamma(s+\frac{d}{2})\Gamma(\frac{d}{2})}{\Gamma(d-1)}
					\frac{b_{s}^2(X)}{(n_1+s+d)^{s+2}n_1^{s}},
		\end{equation*}
with the constant \(n_1\) given by 
$n_1 = 1 + \frac{4}{\delta^2} \left[\frac{2N\Gamma\left(\frac{d-1}{2}\right)}{\sqrt{\pi}}\right]^{\frac{2}{d-2}}$.
\item[{\rm(ii)}] Let \(\epsilon_x := \min\{|x_i-x_j|:i\neq j\in[N_1]\}\) and \(\epsilon_y := \min\{|y_i-y_j|:i\neq j\in[N_2]\}\). Then, for \(d\geq 2, s\geq 3\),  \(\lambda_{\min}(\Theta^\infty_{\net}(X))\) is upper bounded: 
		\begin{equation*} % \label{E:min-upper}
        \begin{aligned}
			\lambda_{\min}(\Theta^\infty_{\net}(X)) & \leq \min \left\{
				\frac{2^{s}s}{\sqrt{\pi}}\frac{\Gamma(s+\frac{d}{2})\Gamma(s-\frac{3}{2})}{\Gamma(s-2+\frac{d}{2})} b_s^2(X), \right. \\
                & \hspace{1cm} \left. 
				2^{s+2}(s-2)^2\frac{\Gamma(s+\frac{d}{2})}{\Gamma(\frac{d}{2})} B^2_{s}(X)\epsilon_x^2, 
				~ 2^{s+1}s^2\frac{\Gamma(s+\frac{d}{2})}{\Gamma(\frac{d}{2})} B^2_{s}(X)\epsilon_y^2
			\right\}.
        \end{aligned}
		\end{equation*}
\end{itemize}
\end{theorem}

\begin{remark}\label{R:bound-min}
	The constraint \(d\geq 3\) for the lower bound stems from the inequality \eqref{E:Ine-Legendre} in Lemma \ref{L:lambda(D_n^TD_n)}, since the Legendre polynomial \(P_{n,2}(t)\) of degree \(n\) in dimension two is oscillatory for \(|t|<1\) as \(n\) grows. 
    There are two main factors that affect the lower bound for the smallest
	eigenvalue, the minimum length of the samples and the distance of their directions. Suppose \(b_s(X)=1\), then the lower bound has the order \(\delta^{2s+2}\) for \(s\) not too large (if we ignore the impact of \(N\)), which shows a decay in the order of the RePU activation \(s\) (RePU order in short). 
	Although the RePU order $s$ will decrease when acted by the differential operators for interior points, the RePU order for boundary points stays unchanged. The order \(\delta^{2s+2}\)  coincides with that in regression task 
	when \(s=0\) \cite[Theorem 1]{Karhadkar:2024Bounds}. 

In Theorem \ref{T:bound-min}, there are three terms in the upper bounds of the smallest eigenvalue with the main factors given by 
	\(b_s(X), \epsilon_x\) and \(\epsilon_y\) respectively. They arise from the following cases 
	that make the DNTK matrix singular: (i) a zero point in samplings points which induces a vanishing row due to the absence of a bias term; (ii) close points in interior or boundary points which induce two nearly identical rows. Note that although we need the \(\delta\)-separated property of the whole sampling points \(X\) to give a lower bound, the DNTK matrix is not necessarily  degenerate once the direction distance is zero, e.g., points with opposite direction, and this is why there is no \(\delta\) appearing in the upper bound; see also \cite{Karhadkar:2024Bounds} for related discussions. 
\end{remark}

\begin{theorem}[Bounds for the largest eigenvalue] \label{T:bound-max}
	Let \(\Theta^\infty_{\net}(X)\) be the infinite DNTK matrix for the 
	Poisson equation \eqref{E:Poisson} with sampling points \(X\).
	Then, for \(d\geq 2, s\geq 3\), \(\lambda_{\max}(\Theta^\infty_{\net}(X))\) satisfies 
	\begin{equation*}
		\frac{2^{s-1}}{\sqrt{\pi}} \Gamma\left(s+\frac{1}{2}\right) B_s^2(X) 
		\leq \lambda_{\max}(\Theta^\infty_{\net}(X)) \leq 
		\frac{2^{s-1}s^2}{\sqrt{\pi}}\frac{\Gamma(s+\frac{d}{2})\Gamma(s-\frac{5}{2})}{\Gamma(s-2+\frac{d}{2})} N B_s^2(X).
	\end{equation*}	
\end{theorem}

Note that the lower bound differ by a factor \(N\) from the upper bound. The former corresponds to the case where all the sampling points are uniformly distributed, while the latter reflects the pathological case that all the points very close to each other so that the largest eigenvalue increases linearly with the number $N$ of sampling points.

\section{The proofs of Theorems \ref{T:bound-min} and \ref{T:bound-max}}
\label{S-A:Bound-eigen}

\subsection{Preliminaries on spherical harmonics}\label{S-A:Sph-Har}
The proofs rely heavily on properties of spherical harmonics. Below we provide a brief overview, following \cite{Atkinson:2012}. Let \(\mathbb{S}^{d-1} = \{ x\in\R^d: |x|=1 \}\) be the unit sphere in the Euclidean space $\mathbb{R}^d$, with $d\geq 2$, 
\(\d S^{d-1}\) be the restriction of the \(d\)-dimensional Lebesgue measure on the sphere
\(\mathbb{S}^{d-1}\), and \(\omega_{d-1}\)  be the area of \(\mathbb{S}^{d-1}\):
\begin{equation}\label{E:area_sphere}
	\omega_{d-1} = \int_{\mathbb{S}^{d-1}}\d S^{d-1} = \frac{2\pi^{\frac{d}{2}}}{\Gamma(\frac{d}{2})}.
\end{equation}
The space \(L^2(\mathbb{S}^{d-1}):=L^2(\mathbb{S}^{d-1};\d S^{d-1})\) 
consists of real-valued square integrable functions on \(\mathbb{S}^{d-1}\), and 
is a Hilbert space, when equipped with the inner product 
\(\int_{\mathbb{S}^{d-1}} fg  ~ \d S^{d-1}\) for \(f,g\in L^2(\mathbb{S}^{d-1})\). For any \(n\geq 0\), \(\mathbb{Y}^d_n \subset L^2(\mathbb{S}^{d-1})\) denotes the space of 
spherical harmonics of degree $n$ on \(\mathbb{S}^{d-1}\), whose dimension is
\(N_{n,d}\), with  \(N_{0,d} = 1\) for all \(d\geq 2\) and
% \(N_{0,d} = N_{1,1} = 1\) for \( d\geq 1\), \(N_{n,1} = 0 \) for \(n\geq 2\), and
\begin{equation}\label{E:N_{n,d}}
	N_{n,d} = \frac{(2n+d-2)(n+d-3)!}{(d-2)!n!} = \binom{n+d-1}{d-1} - \binom{n+d-3}{d-1}, 
	\quad \text{ for } n\geq 1, d\geq 2.
\end{equation}
Let \(\{ Y_{n,j}\}_{j=1}^{N_{n,d} }\) be an orthonormal basis of \(\mathbb{Y}^d_n\). 
Then the set of disjoint union \(\bigsqcup_{n=0}^\infty\{ Y_{n,j}:j\in[N_{n,d}]\} \) forms a 
complete orthonormal basis of \(L^2(\mathbb{S}^{d-1})\). 
Let $P_{n,d}(t)$, $t\in[-1,1]$ be the Legendre polynomial of degree \(n\) in dimension \(d\) with \(P_{n,d}(1)=1\).
Then, we have  
\(P_{n,d}(-t) = (-1)^n P_{n,d}(t)\) for all \(n\geq 0\) and \(|P_{n,d}(t)|\leq 1\) for all \(t\in[-1,1]\).
Moreover, we have the following useful \emph{Rodrigues representation formula}  
\cite[Theorem 2.23]{Atkinson:2012} (for $d\geq 2$): 
\begin{equation}\label{E:Rodrigues-formula}
	P_{n,d}(t) = (-1)^n R_{n,d} (1-t^2)^{\frac{3-d}{2}} \frac{\d^n}{\d t^n} (1-t^2)^{n+\frac{d-3}{2}}, \quad 
 \text{with } R_{n,d} = \frac{\Gamma(\frac{d-1}{2})}{2^n\Gamma(n+\frac{d-1}{2})}.
\end{equation}
The orthonormal basis \(\{ Y_{n,j}\}_{j=1}^{N_{n,d} }\) satisfy the \emph{addition formula} \cite[Theorem 2.9]{Atkinson:2012}:
\begin{equation}\label{E:addition-formula}
	\sum_{j=1}^{N_{n,d}} Y_{n,j}(x)Y_{n,j}( y)
	= \frac{N_{n,d}}{\omega_{d-1}} P_{n,d}(x\cdot y), \quad \forall n\geq 0,  x,y\in\mathbb{S}^{d-1}.
\end{equation}
For any function \(f(t)\in L^1_{\frac{d-3}{2}}(-1,1)\), namely 
\( \int_{-1}^{1} |f(t)|(1-t^2)^{\frac{d-3}{2}}~\d t < \infty \), and any spheric harmonics \(Y_n\in \mathbb{Y}^d_n\), we have the following \emph{Funk-Hecke formula} \cite[Theorem 2.22]{Atkinson:2012}: 
\begin{equation}\label{E:Funk-Hecke}
\int_{\mathbb{S}^{d-1}} f(x\cdot y)Y_n(y) ~ {\rm d}S^{d-1}(y) = \lambda_n(f) Y_n(x), 
	\quad \forall n\geq 0, x\in\mathbb{S}^{d-1},
\end{equation}
with the coefficient \(\lambda_n(f)\) given by
\begin{equation}\label{E:F-H_coeff}
	\lambda_n(f) = \omega_{d-2} \int_{-1}^{1} f(t)P_{n,d}(t)(1-t^2)^{\frac{d-3}{2}}~{\rm d}t.
\end{equation}

\subsection{Lower bound on the smallest eigenvalue}
Let \(m_{d,s}\) be the \(s\)-th moment of the $\chi^2_d$ random variable:
\begin{equation}\label{E:chi-moment}
	m_{d,s} := \E_{|u|^2\sim \chi^2_d} \left[|u|^{2s}\right] = 2^s\frac{\Gamma(s+\frac{d}{2})}{\Gamma(\frac{d}{2})}.
\end{equation} 
From equations \eqref{E:lambda-The_Sig} and \eqref{E:Pro-trans}, with \(g(u;w)\) given in \eqref{E:g(u;w)}, we have
\begin{equation}\label{E:Lam(The)>|g|} % lambda(Theta)> inf|g(u;w)|
\begin{aligned}
	\lambda_{\min}(\Theta^{\infty}_{\net}(X)) 
	& \geq \inf_{w\in\mathbb{S}^{N-1}}\E_{|u|^2\sim \chi^2_d} \left[|u|^{2s}\right] 
	\E_{u\sim \mathrm{U}(\mathbb{S}^{d-1})}\left[|g(u;w)|^2\right] \\
	& = \frac{m_{d,s}}{\omega_{d-1}}\inf_{w\in\mathbb{S}^{N-1}}\|g(u;w)\|^2_{L^2(\mathbb{S}^{d-1})}.
\end{aligned}
\end{equation}
Using the orthonormal basis \(\left\{ Y_{n,j} \right\}_{j=1}^{N_{n,d}}\) of $Y_n^d$, \(g(u;w)\) can be decomposed as 
\begin{align}\label{eqn:guw-exp} g(u;w) = \sum_{n=0}^{\infty} \sum_{j=1}^{N_{n,d}} g_{n,j}(w) Y_{n,j}(u) \quad \text{and} \quad
\|g(u;w)\|_{L^2(\mathbb{S}^{d-1})}^2 = \sum_{n=0}^{\infty} \sum_{j=1}^{N_{n,d}} |g_{n,j}(w)|^2,
\end{align}
with $g_{n,j}(w) = \int_{\mathbb{S}^{d-1}} g(u;w) Y_{n,j}(u) ~ \d S^{d-1}(u)$.
Note that \(d\geq 2\) and \(\sigma_s(t)\leq 1\) for all \(t\in[-1,1]\). 
Using the Funk-Hecke formula \eqref{E:Funk-Hecke} on the function \(g(u;w)\) gives 
\begin{align*} 
	g_{n,j}(w) & = \int_{\mathbb{S}^{d-1}} \left( \sum_{p=1}^{N_1} w^{I}_p (s)_2|x_p|^{s-2}\sigma_{s-2}(u^\top \ol{x}_p) 
	+ \sum_{q=1}^{N_2} w^{B}_q |y_q|^{s}\sigma_s(u^\top \ol{y}_q) \right) Y_{n,j}(u) ~ {\rm d}S^{d-1}(u)\\
	& = \sum_{p=1}^{N_1} w^{I}_p \lambda_n\left((s)_2|x_p|^{s-2}\sigma_{s-2}\right) Y_{n,j}(\ol{x}_p)
	 + \sum_{q=1}^{N_2} w^{B}_q \lambda_n\left(|y_q|^{s}\sigma_s\right) Y_{n,j}(\ol{y}_q).
\end{align*}
Thus for each integer \(n\geq 0\), we have 
\begin{align*}
	\sum_{j=1}^{N_{n,d}} |g_{n,j}(w)|^2 = \sum_{j=1}^{N_{n,d}} 
	\left| \sum_{p=1}^{N_1}w^{I}_p \lambda_n\left((s)_2|x_p|^{s-2}\sigma_{s-2}\right) Y_{n,j}(\ol{x}_p) 
	+ \sum_{q=1}^{N_2} w^{B}_q \lambda_n\left(|y_q|^{s}\sigma_s\right) Y_{n,j}(\ol{y}_q) \right|^2. 
\end{align*}
For each \(n\), we define a diagonal matrix \(T_n\in \R^{N\times N}\):
\[T_n = \mathrm{diag}(T_{n1},\ldots,T_{nN}) := \mathrm{diag}\left(\lambda_n\left((s)_2|x_p|^{s-2}\sigma_{s-2}\right),
\lambda_n\left(|y_q|^{s}\sigma_s\right)\right), \quad \text{for } p\in[N_1], q\in[N_2], \] 
and the matrix \(D_n\in \R^{N_{n,d}\times N}\) such that
\[ D_n = \begin{bmatrix} Y_{n,j}(\ol{x}_\ell) \end{bmatrix}\quad \text{ for } j\in[N_{n,d}], ~\ell\in[N]. \]
Thus, for any \(w\in\mathbb{S}^{N-1}\), the inner quadratic sum in the identity \eqref{eqn:guw-exp} is given by
\begin{align*}\sum_{j=1}^{N_{n,d}} |g_{n,j}(w)|^2 &= \sum_{j=1}^{N_{n,d}} \left| \sum_{\ell=1}^{N} Y_{n,j}(\ol{x}_\ell) T_{n\ell} w_\ell \right|^2 
	= |D_n T_n w|^2 \\ &=  w^\top T_n(D_n^\top D_n)T_n w\geq \lambda_{\min}(T_n(D_n^\top D_n)T_n).
\end{align*} 

\begin{lemma}\label{L:diag-eigen} % eigenvalues of a matrix by diagonal scaling
	Let \(T = \mathrm{diga}(t_1,\ldots,t_n)\in\R^{n\times n}\) be diagonal 
	with \(t_{\min} := \min\{|t_1|, \ldots, |t_n|\}\) and \(t_{\max} = \max\{|t_1|, \ldots, |t_n|\}\).
	If \(X\in\R^{n\times n}\) is a positive semidefinite matrix, then there hold 
	\[ \lambda_{\min}(TXT)\geq t_{\min}^2\lambda_{\min}(X)\quad\mbox{and} \quad 
	\lambda_{\max}(TXT)\leq t_{\max}^2\lambda_{\max}(X). \]
\end{lemma}

\begin{proof}
For any \(x\in\R^n\), let \(y:=Tx\). Then \(y_i = t_i x_i\) for \(i\in[n]\) and 
$|y|_2^2 = \sum_{i=1}^{n} t_i^2 x_i^2 \geq t_{\min}^{2}|x|_2^2.$ 
Thus, we have 
\[ x^\top (TXT)x = y^\top X y \geq \lambda_{\min}(X)|y|_2^2 \geq \lambda_{\min}(X)t_{\min}^{2}|x|_2^2.  \]
By Weyl min-max principle, we have 
	\[ \lambda _{\min}(TXT) = \min_{x\neq 0}\frac{x^\top (TXT)x}{x^\top x}
	\geq \min_{x\neq 0}\frac{\lambda_{\min}(X)t_{\min}^{2}x^\top x}{x^\top x} 
	= t_{\min}^{2}\lambda_{\min}(X). \]
	Similarly, for any \(x\in\R^n\), we have
	$|y|_2^2 = \sum_{i=1}^{n} t_i^2 x_i^2 \leq t_{\max}^{2}|x|_2^2.$ Thus, we get 
	\[ x^\top (TXT)x = y^\top X y \leq \lambda_{\max}(X)|y|_2^2 \leq \lambda_{\max}(X)t_{\max}^{2}|x|_2^2,\] 
	which together with Weyl min-max principle implies
	\[ \lambda _{\max}(TXT) = \max_{x\neq 0}\frac{x^\top (TXT)x}{x^\top x}
	\leq  \max_{x\neq 0}\frac{\lambda_{\max}(X)t_{\max}^{2}x^\top x}{x^\top x} 
	= t_{\max}^{2}\lambda_{\max}(X). \]
This completes the proof of the lemma.
\end{proof}

By Lemma \ref{L:diag-eigen}, we have for any \(w\in \mathbb{S}^{N-1}\) and any integer \(n\geq 0\) 
\begin{align*}
&\sum_{j=1}^{N_{n,d}} |g_{n,j}(w)|^2 \geq \lambda_{\min}(T_n(D_n^\top D_n)T_n) \\
\geq&\min_{p\in[N_1],q\in[N_2]}\left\{ \lambda_n^2\left((s)_2|x_p|^{s-2}\sigma_{s-2}\right), 
\lambda_n^2\left(|y_q|^{s}\sigma_s\right) \right\} \lambda_{\min}(D_n^\top D_n).
\end{align*}
By the linearity of the expression \eqref{E:F-H_coeff}, we have 
\[ \lambda_n^2\left((s)_2|x_p|^{s-2}\sigma_{s-2}\right) = (s)_2^2|x_p|^{2(s-2)}\lambda_n^2(\sigma_{s-2})
\geq b_{s}^2(X)\lambda_n^2(\sigma_{s-2}),\] 
and similarly 
$\lambda_n^2\left(|y_q|^{s}\sigma_s\right)\geq b_{s}^2(X)\lambda_n^2(\sigma_{s})$.
Thus, we arrive at
\begin{equation}\label{E:sum>lambda*lambda}
	\sum_{j=1}^{N_{n,d}} |g_{n,j}(w)|^2 \geq b_{s}^2(X) 
	\min\left\{ \lambda_n^2(\sigma_{s-2}), \lambda_n^2(\sigma_s) \right\}
	\lambda_{\min}(D_n^\top D_n).
\end{equation}

Next, we bound the relevant factors in \eqref{E:sum>lambda*lambda} separately. First we  evaluate the coefficients \(\lambda_n(\sigma_s)\).
The next result has been already derived in \cite[Appendix D.2]{Bach:2017breaking}.
We also provide a proof for the convenience of readers. 

\begin{lemma}\label{L:lambda_n(sigma_s)}
	Let \(\lambda_{n,s} = \lambda_n(\sigma_s)\) for \(n,s\geq 0\). Then the following statements hold.
	\begin{itemize}
		\item for \(n\leq s\), \(\lambda_n(\sigma_s)\neq 0\);
		\item for \(n>s, n\equiv s \!\mod 2\), \(\lambda_n(\sigma_s) = 0\);
		\item for \(n>s, n\not\equiv s \!\mod 2\), 
\begin{equation}\label{E:lambda_n(sigma_s)}
			\lambda_n(\sigma_s) = (-1)^{\frac{n-s-1}{2}}\omega_{d-2}
			\frac{s!\Gamma(\frac{d-1}{2})}{2^{s+1}\sqrt{\pi}}
			\frac{\Gamma(\frac{n-s}{2})}{\Gamma(\frac{n+s+d}{2})}.
		\end{equation}
	\end{itemize}
\end{lemma}

\begin{proof}
By Rodrigues representation formula \eqref{E:Rodrigues-formula} for Legendre polynomials, we have
\begin{align*}
\lambda_n(\sigma_s) &= \omega_{d-2} \int_{-1}^{1} \sigma_s(t)P_{n,d}(t)(1-t^2)^{\frac{d-3}{2}}~{\rm d}t\\
&= (-1)^n \omega_{d-2} R_{n,d} \int_{0}^{1} t^s \left(\frac{\d}{\d t}\right)^n (1-t^2)^{n+\frac{d-3}{2}}~{\rm d}t. 
\end{align*}
By integrating by parts once, we have
\begin{align*}
	\int_{0}^{1} t^s \left(\frac{\d}{\d t}\right)^n (1-t^2)^{n+\frac{d-3}{2}}~\d t
	& = \int_{0}^{1} t^s \frac{\d }{\d t} \left(\frac{\d}{\d t}\right)^{n-1} (1-t^2)^{n+\frac{d-3}{2}} \\
	& = -s \int_{0}^{1} t^{s-1} \left(\frac{\d}{\d t}\right)^{n-1} (1-t^2)^{n+\frac{d-3}{2}}~{\rm d}t.
\end{align*}
If \(n\leq s\), then integrating by parts \(n\) times gives
\[ 
\lambda_n(\sigma_s)  = (-1)^n \omega_{d-2} R_{n,d} (-1)^n (s)_n
\int_{0}^{1} t^{s-n} (1-t^2)^{n+\frac{d-3}{2}}~\d t > 0. 
\]
If \(n> s\), then integration by parts for \(n\) times gives
	\begin{align*}
		\lambda_n(\sigma_s) & = (-1)^n \omega_{d-2} R_{n,d} (-1)^s s!
		\int_{0}^{1} \left(\frac{\d}{\d t}\right)^{n-s} (1-t^2)^{n+\frac{d-3}{2}}~\d t \\
		& = (-1)^{n+s+1} \omega_{d-2} R_{n,d} s! 
		\left. \left(\frac{\d}{\d t}\right)^{n-s-1} (1-t^2)^{n+\frac{d-3}{2}} \right|_{t=0}.
	\end{align*}
	For \(n\not\equiv s \!\mod 2\), using series expansion of $(1-t^2)^{n+\frac{d-3}{2}}$, we have 
	\begin{align*}
		\left. \left(\frac{\d}{\d t}\right)^{n-s-1}(1-t^2)^{n+\frac{d-3}{2}} \right|_{t=0}
		& = \left.  \left(\frac{\d}{\d t}\right)^{n-s-1}
		\sum_{i=0}^{\infty}\binom{n+\frac{d-3}{2}}{i}(-1)^{i} t^{2i}\right|_{t=0} \\
		& = (-1)^{\frac{n-s-1}{2}} \binom{n+\frac{d-3}{2}}{\frac{n-s-1}{2}}(n-s-1)!,
	\end{align*}
	while it vanishes for \(n\equiv s \!\mod 2\).
	By Legendre duplication formula for the Gamma function $\Gamma(z)$, i.e.,  
$\Gamma (z)\Gamma \left(z+{\tfrac {1}{2}}\right)=2^{1-2z}\;{\sqrt {\pi }}\;\Gamma (2z)$, and the 
	formula \(\binom{\alpha}{\beta} = \frac{\Gamma(\alpha+1)}{\Gamma(\beta+1)\Gamma(\alpha-\beta+1)}\), we have 
	\begin{align*}
		R_{n,d}\binom{n+\frac{d-3}{2}}{\frac{n-s-1}{2}} (n-s-1)!
		& = \frac{\Gamma(\frac{d-1}{2})}{2^n\Gamma(n+\frac{d-1}{2})}
		\frac{\Gamma(n+\frac{d-1}{2})\Gamma(n-s)}{\Gamma(\frac{n-s+1}{2})\Gamma(\frac{n+s+d}{2})}\\
		& = \frac{\Gamma(n-s)}{2^n \Gamma(\frac{n-s+1}{2})}
		\frac{\Gamma(\frac{d-1}{2})}{\Gamma(\frac{n+s+d}{2})} 
		=  \frac{\Gamma(\frac{n-s}{2})}{2^{s+1}\sqrt{\pi}}
		\frac{\Gamma(\frac{d-1}{2})}{\Gamma(\frac{n+s+d}{2})}.
	\end{align*}
Combining these identities gives the identity \eqref{E:lambda_n(sigma_s)} for \(n>s, n\not\equiv s \!\mod 2\).
\end{proof}

Next, we give upper and lower bounds on the matrix \(D_n^\top D_n\). 
\begin{lemma}\label{L:lambda(D_n^TD_n)}
Given the nonzero sampling points \(X=\{x_\ell\}_{\ell=1}^N\), \(\lambda_{\max}(D_n^\top D_n)\) satisfies
\[ \lambda_{\max}\left(D_n^\top D_n\right) \leq \frac{N N_{n,d}}{\omega_{d-1}}, \quad \text{ for } ~d\geq 2, ~ n\geq 0. 
\]
Moreover, if \(X\) is \(\delta\)-separated, then  
\[ \lambda_{\min}\left(D_n^\top D_n\right) \geq \frac{N_{n,d}}{2\omega_{d-1}}, \quad \text{ for } 
	~ d\geq 3, ~ n\geq K_1:=\frac{4}{\delta^2}\left[\frac{2N\Gamma\left(\frac{d-1}{2}\right)}{\sqrt{\pi}}\right]^{\frac{2}{d-2}}. \]
\end{lemma}
\begin{proof}
Note that the \((i,\ell)\)-entry of the matrix \(D_n^\top D_n\) is the inner product of 
the \(i\)-th and \(\ell\)-th rows of \(D_n\). By the addition formula \eqref{E:addition-formula} for spherical harmonics, we have
\[ D_n^\top D_n(i,\ell) = \sum_{j=1}^{N_{n,d}} Y_{n,j}(\ol{x}_i)Y_{n,j}(\ol{x}_\ell)
	= \frac{N_{n,d}}{\omega_{d-1}} P_{n,d}(\ol{x}_i\cdot\ol{x}_\ell), \quad i, \ell\in [N]. \]  
Note that
\( P_{n,d}(1) = 1\) and \(|P_{n,d}(t)|\leq 1\), \(t\in[0,1] \) for all \(d\geq 2\) and \(n\geq 0\).
Thus \(|P_{n,d}(\ol{x}_i\cdot\ol{x}_\ell)|\leq 1\) for all \(i,\ell\in[N]\). For each \(i\in[N]\), we have 
\begin{align*}
	D_n^\top D_n(i,i) = \frac{N_{n,d}}{\omega_{d-1}} P_{n,d}(|\ol{x}_i|^2)
	= \frac{N_{n,d}}{\omega_{d-1}} P_{n,d}(1) = \frac{N_{n,d}}{\omega_{d-1}}.
\end{align*} 
By the Gershgorin circle theorem,
the largest eigenvalue of \(D_n^\top D_n\) satisfies
\begin{align*}
	\lambda_{\max}\left(D_n^\top D_n\right) 
	\leq& \max_{i\in[N]} \Big\{ D_n^\top D_n(i,i) + \sum_{\ell\neq i} 
	\big|D_n^\top D_n(i,\ell)\big| \Big\} 
	 \leq \frac{N_{n,d}}{\omega_{d-1}} + (N-1)\frac{N_{n,d}}{\omega_{d-1}}
	= \frac{N N_{n,d}}{\omega_{d-1}}. 
\end{align*}
To obtain a lower bound on \(\lambda_{\min}\left(D_n^\top D_n\right)\), we use the following inequality for 
the Legendre polynomial $P_{n,d}(t)$ for the points away from the ends \cite[(2.117)]{Atkinson:2012}:
\begin{equation}\label{E:Ine-Legendre}
	\left| P_{n,d}(t) \right| < \frac{\Gamma\left(\frac{d-1}{2}\right)}{\sqrt{\pi}} 
\left[ \frac{4}{n \left(1 - t^2\right)} \right]^{\frac{d-2}{2}}, \quad n \geq 0, 
 d\geq 3, t \in (-1,1).
\end{equation}
For \(\ell\neq i\), by the \(\delta\)-separated assumption \eqref{E:delta-sepa}, 
we have \(1 - \llan \ol{x}_i,\ol{x}_\ell \rran ^2 \geq \delta^2\) and thus 
\begin{align*}
	\big|D_n^\top D_n(i,\ell)\big| \leq& \frac{N_{n,d}}{\omega_{d-1}}
	\frac{\Gamma\left(\frac{d-1}{2}\right)}{\sqrt{\pi}} 
	\left[ \frac{4}{1 - \llan \ol{x}_i,\ol{x}_\ell \rran ^2} \right]^{\frac{d-2}{2}}
	\left(\frac{1}{n}\right)^{\frac{d-2}{2}} 
	\leq \frac{N_{n,d}}{\omega_{d-1}}
	\frac{\Gamma\left(\frac{d-1}{2}\right)}{\sqrt{\pi}} 
	\left( \frac{2}{\delta} \right)^{d-2}\left(\frac{1}{n}\right)^{\frac{d-2}{2}}. 
\end{align*}
By the Gershgorin circle theorem again, the smallest eigenvalue \(\lambda_{\min}(D_n^\top D_n)\) satisfies
\begin{align*}
	\lambda_{\min}\left(D_n^\top D_n\right) 
	& \geq \min_{i\in[N]} \left\{ D_n^\top D_n(i,i) - \sum_{\ell\neq i} 
	\big|D_n^\top D_n(i,\ell)\big| \right\} \\
	& \geq \frac{N_{n,d}}{\omega_{d-1}} - (N-1)\frac{N_{n,d}}{\omega_{d-1}}
	\frac{\Gamma\left(\frac{d-1}{2}\right)}{\sqrt{\pi}} 
	\left( \frac{2}{\delta} \right)^{d-2}\left(\frac{1}{n}\right)^{\frac{d-2}{2}}. 
\end{align*}
The conclusion is obtained if \(n\) is large enough such that 
\[ N\frac{\Gamma\left(\frac{d-1}{2}\right)}{\sqrt{\pi}} 
	\left( \frac{2}{\delta} \right)^{d-2}\left(\frac{1}{n}\right)^{\frac{d-2}{2}}\leq \frac{1}{2}
	\quad \Longleftrightarrow \quad n\geq K_1 := \frac{4}{\delta^2}
\left[\frac{2N\Gamma\left(\frac{d-1}{2}\right)}{\sqrt{\pi}}\right]^{\frac{2}{d-2}}. \]
\end{proof}

Next we bound \(\min\left\{ \lambda_n^2(\sigma_{s-2}), \lambda_n^2(\sigma_s) \right\}\) for \(s\geq 3\) 
and \(n>s, n\not\equiv s \!\mod 2\). Let \(I_{n,s}\) be the ratio of Gamma functions arising in equation \eqref{E:lambda_n(sigma_s)} whose arguments involve \(n\):
\begin{equation}\label{E:I_{n,s}}
	I_{n,s} = \frac{\Gamma(\frac{n-s}{2})}{\Gamma(\frac{n+s+d}{2})}.
\end{equation}
Note that \(\Gamma(x)\) is increasing in the interval \([2,\infty)\). When choosing \(n\) large enough 
as in Lemma \ref{L:lambda(D_n^TD_n)}, \(I_{n,s}\) is decreasing with respect to \(s\) if \(s\leq n-4\).
By the formula \eqref{E:lambda_n(sigma_s)}, we have 
\begin{equation}\label{E:lambda_n(s,s-2)}
	|\lambda_n(\sigma_s)| = \omega_{d-2}\frac{\Gamma(\frac{d-1}{2})}{\sqrt{\pi}}
	\frac{s!}{2^{s+1}} I_{n,s}\quad\mbox{and} \quad 
	\left|\lambda_n(\sigma_{s-2})\right| = \omega_{d-2}\frac{\Gamma(\frac{d-1}{2})}{\sqrt{\pi}}
	\frac{(s-2)!}{2^{s-1}} I_{n,s-2}.
\end{equation}
Note that the ratio \(\frac{s!}{2^{s+1}}\) is increasing for \(s\geq 1\). 
Thus, for \(n>s, n\not\equiv s \!\mod 2\), we have 
\begin{align*}
	\min\left\{ \lambda_n^2(\sigma_{s-2}), \lambda_n^2\left(\sigma_s\right) \right\}
	\geq \omega_{d-2}^2\frac{\Gamma(\frac{d-1}{2})^2}{\pi}
	\frac{((s-2)!)^2}{2^{2s-2}} I_{n,s}^2.
\end{align*}
From equation \eqref{E:sum>lambda*lambda}, 
for \(n\geq K_1>s, ~ n\not\equiv s \!\mod 2\) (with $K_1$ given in Lemma \ref{L:lambda(D_n^TD_n)}) and any \(w\in\mathbb{S}^{N-1}\), we have  
\begin{align}
\sum_{j=1}^{N_{n,d}} |g_{n,j}(w)|^2 &\geq b_{s}^2(X) \omega_{d-2}^2\frac{\Gamma(\frac{d-1}{2})^2}{\pi}
	\frac{((s-2)!)^2}{2^{2s-2}} I_{n,s}^2 \frac{N_{n,d}}{2\omega_{d-1}} \nonumber\\
	&= b_{s}^2(X) \frac{((s-2)!)^2}{2^{2s-1}}\frac{\omega_{d-2}^2}{\omega_{d-1}}\frac{\Gamma(\frac{d-1}{2})^2}{\pi} 
	I_{n,s}^2N_{n,d}. \label{E:sum>I_{n,s}^2*lambda}
\end{align}

Next we estimate the factor \(I_{n,s}^2N_{n,d}\). We first recall two bounds on the Gamma function. 
The first can be found in \cite[Theorem 1]{Jameson:2015simple}, and the second in \cite{Robbins:1955remark}.
\begin{lemma}\label{L:bnd-Gamma}
	For all \(x>0\) and \(1\leq n\in \N\), we have 
	\begin{align*}
		\sqrt{2\pi}x^{x-\frac{1}{2}}e^{-x}\leq & ~\Gamma(x)~ 
		\leq \sqrt{2\pi}x^{x-\frac{1}{2}}e^{-x}e^{\frac{1}{12x}}, \\
		\sqrt{2\pi}n^{n+\frac{1}{2}}e^{-n}\leq & ~n!~
		\leq \sqrt{2\pi}n^{n+\frac{1}{2}}e^{-n}e^{\frac{1}{12n}}.
	\end{align*}
\end{lemma}

\begin{lemma}\label{L:I_{n,s}^2*N_{n,d}}
	Let \(I_{n,s}\) be defined by \eqref{E:I_{n,s}} and \(N_{n,d}\) given by \eqref{E:N_{n,d}}
	for \(s\geq 0, d\geq 2\) and \(n\geq s+1\), i.e.,  
	\[ I_{n,s} = \frac{\Gamma(\frac{n-s}{2})}{\Gamma(\frac{n+s+d}{2})} \quad \text{ and } \quad
	N_{n,d} = \frac{(2n+d-2)(n+d-3)!}{(d-2)!n!}. \]
	Then the following inequalities hold
	\[  \frac{2^{2s+d+1}}{(d-2)!}\frac{e^{-1}}{(n+s+d)^{s+2}n^{s}}
		\leq I_{n,s}^2 N_{n,d}\leq 
		\frac{2^{2s+d+1}}{(d-2)!}\frac{e^{2s+4}}{(n+s+d)^{s}n^{s+2}}. \]
\end{lemma}

\begin{proof}
	
If \(n+d-3=0\), we must have \(d=2,n=1\) and \(s=0\). Then, we have \(I_{n,s}=\frac{1}{2}\), \(N_{n,d}=2\), 
and the inequalities can be verified directly. Next, suppose \(n+d-3>0\).
By Lemma \ref{L:bnd-Gamma}, we have 
\begin{align*}
	\frac{\frac{4\pi}{n-s}\left( \frac{n-s}{2e} \right)^{n-s} }
	{\frac{4\pi}{n+s+d}\left( \frac{n+s+d}{2e} \right)^{n+s+d} e^{\frac{1}{12}}}
	\leq  I_{n,s}^2 = \frac{\Gamma(\frac{n-s}{2})^2}{\Gamma(\frac{n+s+d}{2})^2} \leq 
	\frac{\frac{4\pi}{n-s}\left( \frac{n-s}{2e} \right)^{n-s}e^{\frac{1}{3}}}
	{\frac{4\pi}{n+s+d}\left( \frac{n+s+d}{2e} \right)^{n+s+d}}, 
\end{align*}
namely, 
\[ \frac{(n-s)^{n-s-1}(2e)^{2s+d}e^{-\frac{1}{12}}}{(n+s+d)^{n+s+d-1}}
	\leq   I_{n,s}^2 \leq \frac{(n-s)^{n-s-1}(2e)^{2s+d}e^{\frac{1}{3}}}{(n+s+d)^{n+s+d-1}}. \]
Similarly, for \(n+d-3\geq 1\), we have 
\begin{align*}
	\frac{\sqrt{2\pi(n+d-3)}\left(\frac{n+d-3}{e}\right)^{n+d-3}}
	{\sqrt{2\pi n}\left(\frac{n}{e}\right)^{n}e^{\frac{1}{6}}}
	\leq & ~ \frac{(n+d-3)!}{n!} \leq 
	\frac{\sqrt{2\pi(n+d-3)}\left(\frac{n+d-3}{e}\right)^{n+d-3}e^{\frac{1}{12}}}
	{\sqrt{2\pi n}\left(\frac{n}{e}\right)^{n}}, 
\end{align*}
namely,
\begin{align*}	
\frac{(n+d-3)^{n+d-\frac{5}{2}}e^{-\frac{1}{6}}}{n^{n+\frac{1}{2}}e^{d-3}}
	\leq & ~ \frac{(n+d-3)!}{n!} \leq \frac{(n+d-3)^{n+d-\frac{5}{2}}e^{\frac{1}{12}}}{n^{n+\frac{1}{2}}e^{d-3}}.
\end{align*}
Combining these inequalities for $I_{n,s}$ and $N_{n,d}$ gives 
\begin{align*}
	I_{n,s}^2 N_{n,d} & \geq \frac{(n-s)^{n-s-1}(2e)^{2s+d}e^{-\frac{1}{12}}}{(n+s+d)^{n+s+d-1}}
	\frac{(2n+d-2)}{(d-2)!}\frac{(n+d-3)^{n+d-\frac{5}{2}}e^{-\frac{1}{6}}}{n^{n+\frac{1}{2}}e^{d-3}}\\
	& = \frac{2^{2s+d+1}e^{2s+\frac{11}{4}}}{(d-2)!}
	\frac{n+\frac{d}{2}-1}{(n+s+d)^{s+2}n^{s+1}}\sqrt{\frac{n+d-3}{n}}
	\left(\frac{n-s}{n}\right)^{n-s-1}\left(\frac{n+d-3}{n+s+d}\right)^{n+d-3} \\
	& \geq \frac{2^{2s+d+1}e^{2s+\frac{9}{4}}}{(d-2)!}\frac{1}{(n+s+d)^{s+2}n^{s}}
	e^{-s}e^{-(s+3)} = \frac{2^{2s+d+1}}{(d-2)!}\frac{e^{-\frac{3}{4}}}{(n+s+d)^{s+2}n^{s}},
\end{align*}
using \(n+\frac{d}{2}-1\geq n\), \(e(n+d-3)\geq n\) and the limit of the decreasing sequence
$(\frac{n-s}{n})^{n-s} = (1 + \frac{s}{n-s})^{-(n-s)}
\xrightarrow{n\to \infty}  e^{-s}$.
Similarly, we have 
\begin{align*}
	& I_{n,s}^2 N_{n,d} \\
    \leq &  \frac{2^{2s+d+1}e^{2s+3+\frac{5}{12}}}{(d-2)!}
	\frac{n+\frac{d}{2}-1}{(n+s+d)^{s+2}n^{s+1}}\sqrt{\frac{n+d-3}{n}}
	\left(\frac{n-s}{n}\right)^{n-s-1}\left(\frac{n+d-3}{n+s+d}\right)^{n+d-3} \\
	\leq &  \frac{2^{2s+d+1}}{(d-2)!}\frac{e^{2s+3+\frac{5}{12}}}{(n+s+d)^{s+\frac{1}{2}}n^{s+\frac{3}{2}}}
	\leq \frac{2^{2s+d+1}}{(d-2)!}\frac{e^{2s+3+\frac{5}{12}}}{(n+s+d)^{s}n^{s+2}}.
\end{align*}
\end{proof}

By combining \eqref{E:Lam(The)>|g|} and \eqref{E:sum>I_{n,s}^2*lambda}, we get for \(s\leq K_1-4\) (with \(K_1\) given in Lemma \ref{L:lambda(D_n^TD_n)}),  
\begin{align*}
	\lambda_{\min}(\Theta^{\infty}_{\net}(X)) 
	& \geq \frac{m_{d,s}}{\omega_{d-1}}\inf_{w\in\mathbb{S}^{N-1}}\|g(u;w)\|_{L^2(\mathbb{S}^{d-1})}
	  = \frac{m_{d,s}}{\omega_{d-1}}\inf_{w\in\mathbb{S}^{N-1}}
		\sum_{n=0}^{\infty} \sum_{j=1}^{N_{n,d}} |g_{n,j}(w)|^2 \\
	& \geq \frac{m_{d,s}}{\omega_{d-1}} \sum_{n\geq K_1, n\not\equiv s}
		b_{s}^2(X) \frac{((s-2)!)^2}{2^{2s-1}}\frac{\omega_{d-2}^2}{\omega_{d-1}}\frac{\Gamma(\frac{d-1}{2})^2}{\pi} 
		I_{n,s}^2N_{n,d} \\
	& = m_{d,s}b_{s}^2(X) \frac{((s-2)!)^2}{2^{2s-1}} \frac{\omega_{d-2}^2}{\omega_{d-1}^2}
		\frac{\Gamma(\frac{d-1}{2})^2}{\pi} \sum_{n\geq K_1, n\not\equiv s} I_{n,s}^2 N_{n,d}.
\end{align*}
Let \(n_1 = 1+K_1\). Then we have
\[ \sum_{n=K_1, n\not\equiv s}^{\infty}I_{n,s}^2 N_{n,d} \geq I_{n_1,s}^2 N_{n_1,d} 
\geq \frac{2^{2s+d+1}}{(d-2)!}\frac{e^{-1}}{(n_1+s+d)^{s+2}n_1^{s}}. \]
By the identities \eqref{E:chi-moment} and \eqref{E:area_sphere}, we obtain 
\begin{align*}
	\lambda_{\min}(\Theta^{\infty}_{\net}(X)) 
	& \geq m_{d,s}b_{s}^2(X) \frac{((s-2)!)^2}{2^{2s-1}} \frac{\omega_{d-2}^2}{\omega_{d-1}^2}
		\frac{\Gamma(\frac{d-1}{2})^2}{\pi} \frac{2^{2s+d+1}}{(d-2)!} \frac{e^{-1}}{(n_1+s+d)^{s+2}n_1^{s}} \\
	& = 2^s\frac{\Gamma(s+\frac{d}{2})}{\Gamma(\frac{d}{2})} b_{s}^2(X) \frac{\Gamma(s-1)^2}{e\pi} 
		\frac{4\pi^{d-1}}{\Gamma(\frac{d-1}{2})^2} \frac{\Gamma(\frac{d}{2})^2}{4\pi^{d}}
		\frac{\Gamma(\frac{d-1}{2})^2}{\Gamma(d-1)} \frac{2^{d+2}}{(n_1+s+d)^{s+2}n_1^{s}} \\
	& = \frac{2^{s+d+2}}{e\pi^2} \frac{\Gamma(s-1)^2\Gamma(s+\frac{d}{2})\Gamma(\frac{d}{2})}{\Gamma(d-1)}
		\frac{b_{s}^2(X)}{(n_1+s+d)^{s+2}n_1^{s}}.
\end{align*}

\subsection{Upper bound on the smallest eigenvalue}
To derive an upper bound on the smallest eigenvalue \(\lambda_{\min}(\Theta^{\infty}_{\net}(X))\), we analyze the two terms 
in \eqref{E:Inner-DNTK} of \(w^\top \Theta^{\infty}_{\net}(X) w\). For the first term,  by \eqref{E:Pro-trans},
we have 
\begin{align*}
	w^\top \Sigma^{\infty}_{\net}(X) w = 
	m_{d,s} \E_{u\sim \mathrm{U}(\mathbb{S}^{d-1})}\left[|g(u;w)|^2\right], \quad w\in \mathbb{S}^{d-1}.
\end{align*}
It suffices to choose some \(w\in \mathbb{S}^{d-1}\) to make \(\|g(u;w)\|_{L^2(\mathbb{S}^{d-1})}\) as small as possible. Note that there are two cases of the choice of sampling points \(X = X_{I} \sqcup X_{B}\) that make the DNTK matrix 
\(\Theta^{\infty}_{\net}(X)\) nearly singular. The first case is that there is some point in \(X\) nearly zero since a zero point leads to a vanishing row in the Gram matrix \(\Theta^{\infty}_{\net}(X)\) (due to the absence of the bias term). The second case is that there are two 
points in \(X_I\) or \(X_B\) very close to each other. In the first case, we choose \(w=(1,0,\ldots,0)\)
suppose \(x_1\in \arg\min\left\{ (s)_2^2|x_p|^{2(s-2)}, |y_q|^{2s} \right\}\).  
Then, from the expression \eqref{E:g(u;w)}  of \(g(u;w)\), we have
\begin{align*}
	w^\top \Sigma^{\infty}_{\net}(X) w & = m_{d,s}
	\E_{u\sim \mathrm{U}(\mathbb{S}^{d-1})}\left[ \left| (s)_2 |x_1|^{s-2} \sigma_{s-2}(u^\top \ol{x}_1) \right|^2\right] \\
	& = m_{d,s} b_s^2(X)\E_{u\sim \mathrm{U}(\mathbb{S}^{d-1})}\left[ \left| \sigma_{s-2}(u^\top \ol{x}_1) \right|^2\right].
\end{align*}
By symmetry, we can assume \(u^\top \ol{x}_1 = u_1 = \cos\theta\). 
Since \(\d S^{d-1}(u) = (\sin \theta)^{d-2} \d\theta \d S^{d-2}\), using \eqref{E:area_sphere} and the relation 
$2\int_{0}^{\frac{\pi}{2}} (\cos\theta)^{2a-1} (\sin \theta)^{2b-1} ~ \d\theta
= \frac{\Gamma(a)\Gamma(b)}{\Gamma(a+b)}$, we have
\begin{align}\label{E:E[sigma(ux_1)^2]}
	&\E_{u\sim \mathrm{U}(\mathbb{S}^{d-1})}\left[ \left| \sigma_s(u^\top \ol{x}_1) \right|^2\right]
	 = \E_{u\sim \mathrm{U}(\mathbb{S}^{d-1})}\left[ (u_1)_+^{2s}\right]
	  = \frac{1}{\omega_{d-1}}\int_{\mathbb{S}^{d-1}} (u_1)_+^{2s} ~ \d S^{d-1}(u) \nonumber\\
	 =& \frac{1}{\omega_{d-1}}\int_{0}^{\pi} (\cos\theta)_+^{2s} ~ (\sin \theta)^{d-2} \d\theta \d S^{d-2} 
	 = \frac{\omega_{d-2}}{\omega_{d-1}} 
		\int_{0}^{\frac{\pi}{2}} (\cos\theta)^{2s} (\sin \theta)^{d-2} ~ \d\theta \nonumber\\
	 =& \frac{2\pi^{\frac{d-1}{2}}}{\Gamma(\frac{d-1}{2})}\frac{\Gamma(\frac{d}{2})}{2\pi^{\frac{d}{2}}}\cdot
		\frac{1}{2}\frac{\Gamma(s+\frac{1}{2})\Gamma(\frac{d-1}{2})}{\Gamma(s+\frac{d}{2})} 
	 = \frac{1}{2\sqrt{\pi}}\frac{\Gamma(\frac{d}{2})\Gamma(s+\frac{1}{2})}{\Gamma(s+\frac{d}{2})}.
\end{align}
Therefore, we derive 
\begin{align*}
	\lambda_{\min}\left(\Sigma^\infty_{\net}(X)\right) 
	\leq 2^s\frac{\Gamma(s+\frac{d}{2})}{\Gamma(\frac{d}{2})} 
		b_s^2(X) \frac{1}{2\sqrt{\pi}}\frac{\Gamma(\frac{d}{2})\Gamma(s-\frac{3}{2})}{\Gamma(s-2+\frac{d}{2})}
	= \frac{2^{s-1}}{\sqrt{\pi}}\frac{\Gamma(s+\frac{d}{2})\Gamma(s-\frac{3}{2})}{\Gamma(s-2+\frac{d}{2})} b_s^2(X).
\end{align*}
For the second term in \eqref{E:Inner-DNTK}, the function \(g_k(u;w)\) for \(k\in[d]\) 
given in \eqref{E:g_k(u;w)} satisfies
\begin{align*}
	g_k(u;w) &= \Delta\left( \sigma_s'(u^\top x_1)x_{1k} \right) 
	= x_{1k}\sigma_s^{(3)}(u^\top x_1)|u|^2 + 2\sigma_s^{(2)}(u^\top x_1)u_k, \\
	|g_k(u;w)|^2 &\leq 2x_{1k}^2\sigma_s^{(3)}(u^\top x_1)^2|u|^4 + 4\sigma_s^{(2)}(u^\top x_1)^2u_k^2.
\end{align*}
By summing the preceding inequality over \(k\in[d]\), we have 
\begin{align*}
	\sum_{k=1}^{d}| g_k(u;w) |^2 \leq 2|x_{1}|^2\sigma_s^{(3)}(u^\top x_1)^2|u|^4 + 4\sigma_s^{(2)}(u^\top x_1)^2|u|^2.
\end{align*}
Therefore, by equation \eqref{E:E[sigma(ux_1)^2]} we arrive at 
\begin{align*}
	 &\sum_{k=1}^{d}\E_{u} \left[ |g_k(u;w)|^2 \right]\\
	= &  \E_{u\sim\NN(0,I)} \left[ \sum_{k=1}^{d} |g_k(u;w)|^2 \right] 
		\leq \E_{u\sim\NN(0,I)} \left[ 2|x_{1}|^2\sigma_s^{(3)}(u^\top x_1)^2|u|^4 + 
		4\sigma_s^{(2)}(u^\top x_1)^2|u|^2 \right] \\
	=& 2(s)_2^2 |x_1|^{2(s-2)} \E_{|u|^2\sim \chi^2_d} \left[|u|^{2(s-1)}\right] 
		\E_{u\sim\mathrm{U}(\mathbb{S}^{d-1})} \left[ (s-2)^2\sigma_{s-3}(u^\top \ol{x}_1)^2 + 
		2\sigma_{s-2}(u^\top \ol{x}_1)^2 \right] \\
=	& 2 b_s^2(X) m_{d,s-1} 
		\left(\frac{(s-2)^2}{2\sqrt{\pi}}\frac{\Gamma(\frac{d}{2})\Gamma(s-\frac{5}{2})}{\Gamma(s-3+\frac{d}{2})}
		+ \frac{2}{2\sqrt{\pi}}\frac{\Gamma(\frac{d}{2})\Gamma(s-\frac{3}{2})}{\Gamma(s-2+\frac{d}{2})}\right) \\
	=& 2 b_s^2(X) 2^{s-1}\frac{\Gamma(s-1+\frac{d}{2})}{\Gamma(\frac{d}{2})}
		\left(\frac{(s-2)^2}{2\sqrt{\pi}}\frac{\Gamma(\frac{d}{2})\Gamma(s-\frac{5}{2})}{\Gamma(s-3+\frac{d}{2})}
		+ \frac{2}{2\sqrt{\pi}}\frac{\Gamma(\frac{d}{2})\Gamma(s-\frac{3}{2})}{\Gamma(s-2+\frac{d}{2})}\right) \\
	=& \frac{2^{s-1}}{\sqrt{\pi}}
		\left(\frac{(s-3+\frac{d}{2})(s-2)^2}{(s-\frac{5}{2})(s-1+\frac{d}{2})} + \frac{2}{s-1+\frac{d}{2}}\right)
		\frac{\Gamma(s+\frac{d}{2})\Gamma(s-\frac{3}{2})}{\Gamma(s-2+\frac{d}{2})} b_s^2(X).
\end{align*}
Since \(\frac{s-3+\frac{d}{2}}{s-1+\frac{d}{2}}, \frac{2}{s-1+\frac{d}{2}}\leq 1\), we have 
\begin{align*}
	\lambda_{\min}\left(\Theta^\infty_{\net}(X)\right) 
	\leq \frac{2^{s-1}}{\sqrt{\pi}} \frac{s^2-2s-1}{s-\frac{5}{2}}
	 \frac{\Gamma(s+\frac{d}{2})\Gamma(s-\frac{3}{2})}{\Gamma(s-2+\frac{d}{2})} b_s^2(X).
\end{align*}
Similarly, if \(y_1\in \arg\min\left\{ (s)_2^2|x_p|^{2(s-2)}, |y_q|^{2s} \right\}\), let \(w=(0,\ldots,0,1,0,\ldots,0)\)
with one lying at the \(N_1+1\)-th position. Then, for the first term in \eqref{E:Inner-DNTK}, we get 
\begin{align*}
	\lambda_{\min}\left(\Sigma^\infty_{\net}(X)\right) \leq w^\top \Sigma^{\infty}_{\net}(X) w 
	= m_{d,s} |y_1|^{2s} \E_{u\sim \mathrm{U}(\mathbb{S}^{d-1})}\left[ \left| \sigma_s(u^\top \ol{y}_1) \right|^2\right] 
	= \frac{2^{s-1}}{\sqrt{\pi}}\Gamma\left(s+\frac{1}{2}\right)b_s^2(X). 
\end{align*}
For the second term, we have 
\begin{align*}
	 \sum_{k=1}^{d}\E_{u} \left[ |g_k(u;w)|^2 \right] 
	& = \E_{u\sim\NN(0,I)} \left[ \sum_{k=1}^{d} |\sigma_s'(u^\top y_1)y_{1k}|^2 \right] 
	  = \E_{u\sim\NN(0,I)} \left[ \sigma_s'(u^\top y_1)^2 |y_{1}|^2 \right] \\
	& = |y_1|^{2s} \E_{|u|^2\sim \chi^2_d} \left[|u|^{2(s-1)}\right] \E_{u\sim \mathrm{U}(\mathbb{S}^{d-1})}
		\left[ s^2 \sigma_{s-1}(u^\top \ol{y}_1)^2 \right] \\
	& = b_s^2(X) 2^{s-1} \frac{\Gamma(s-1+\frac{d}{2})}{\Gamma(\frac{d}{2})} 
		\frac{s^2}{2\sqrt{\pi}}\frac{\Gamma(\frac{d}{2})\Gamma(s-\frac{1}{2})}{\Gamma(s-1+\frac{d}{2})} \\
	& = \frac{2^{s-1}}{\sqrt{\pi}}\frac{s^2}{2s-1}\Gamma\left(s+\frac{1}{2}\right)b_s^2(X).
\end{align*}
Thus, for \(y_1\in \arg\min\left\{ (s)_2^2|x_p|^{2(s-2)}, |y_q|^{2s} \right\}\) we have 
\begin{align*}
	\lambda_{\min}\left(\Theta^\infty_{\net}(X)\right) 
	\leq \frac{2^{s-1}}{\sqrt{\pi}}\frac{s^2+2s-1}{2s-1}\Gamma\left(s+\frac{1}{2}\right)b_s^2(X).
\end{align*}
Note that \(\frac{s^2-2s-1}{s-\frac{5}{2}}\geq \frac{s^2+2s-1}{2s-1}\) for \(s\geq 3\) and
\[\frac{\Gamma(s+\frac{d}{2})\Gamma(s-\frac{3}{2})}{\Gamma(s-2+\frac{d}{2})}\geq \Gamma\left(s+\frac{1}{2}\right)
= \left(s-\frac{1}{2}\right)\left(s-\frac{3}{2}\right)\Gamma\left(s-\frac{3}{2}\right),\]
since \(\frac{\Gamma(s+\frac{d}{2})}{\Gamma(s-2+\frac{d}{2})} = (s-1+\frac{d}{2})(s-2+\frac{d}{2})\) and \(d\geq 2\). 
Therefore, we derive 
\[ \lambda_{\min}\left(\Theta^\infty_{\net}(X)\right) 
	\leq \frac{2^{s-1}}{\sqrt{\pi}} \frac{s^2-2s-1}{s-\frac{5}{2}}
	\frac{\Gamma(s+\frac{d}{2})\Gamma(s-\frac{3}{2})}{\Gamma(s-2+\frac{d}{2})} b_s^2(X)
	\leq \frac{2^{s}s}{\sqrt{\pi}}\frac{\Gamma(s+\frac{d}{2})\Gamma(s-\frac{3}{2})}{\Gamma(s-2+\frac{d}{2})} b_s^2(X). \]
Next, we consider the second case. 
Suppose \(\epsilon_x = |\ol{x}_1-\ol{x}_2| = 
\min\left\{ |\ol{x}_i-\ol{x}_j|: i\neq j\in[N_1] \right\}\). 
Let the test vector \(w=[\alpha,\beta,0,\ldots,0]^\top\) with \(\alpha^2+\beta^2=1\). 
Then, we have  
\begin{align*}
	w^\top \Sigma^{\infty}_{\net}(X) w 
	& = m_{d,s} \E_{u\sim \mathrm{U}(\mathbb{S}^{d-1})} \left[ \left| \alpha
	\sigma_s^{(2)}(u^\top x_1) + \beta \sigma_s^{(2)}(u^\top x_2) \right|^2\right] \\
	& = m_{d,s} \E_{u\sim \mathrm{U}(\mathbb{S}^{d-1})} \left[ \left| 
		\alpha (s)_2|x_1|^{s-2}\sigma_{s-2}(u^\top \ol{x}_1) + 
		\beta (s)_2|x_2|^{s-2}\sigma_{s-2}(u^\top \ol{x}_2) \right|^2\right].
\end{align*}
We can choose \(\alpha,\beta\) so that 
\(\alpha (s)_2|x_1|^{s-2} = -\beta (s)_2|x_2|^{s-2} \), for which we also have 
\[|\alpha (s)_2|x_1|^{s-2}|\leq \max\left\{ (s)_2|x_p|^{s-2}, |y_q|^{s} \right\}=B_s(X).\]
Then we arrive at 
\begin{align*}
	w^\top \Sigma^{\infty}_{\net}(X) w & \leq m_{d,s} B_s^2(X)
	\E_{u\sim \mathrm{U}(\mathbb{S}^{d-1})} \left[ | 
	\sigma_{s-2}(u^\top \ol{x}_1) - \sigma_{s-2}(u^\top \ol{x}_2)|^2\right] \\
	& \leq m_{d,s} B_s^2(X)
	\E_{u\sim \mathrm{U}(\mathbb{S}^{d-1})} \left[ (s-2)^2 | 
	u^\top (\ol{x}_1 - \ol{x}_2)|^2\right] 
\end{align*}
using the estimate \(|\sigma_s(x)-\sigma_s(y)|\leq s|x-y| \) for \(s\geq 1\) and \(|x|,|y|\leq 1\).
For \(x\in \mathbb{S}^{d-1}\), by the symmetry of the integral, we have
\begin{align*}
	\E_{u\sim \mathrm{U}(\mathbb{S}^{d-1})} \left[ |u^\top x|^2\right]
	= \E_{u\sim \mathrm{U}(\mathbb{S}^{d-1})} \left[ u_1^2\right]
	= \frac{1}{d}\sum_{k=1}^{d} \E_{u\sim \mathrm{U}(\mathbb{S}^{d-1})} \left[ u_k^2\right]
	= \frac{1}{d}.
\end{align*}
Therefore, we have 
\begin{align*}
	\lambda_{\min}\left(\Sigma^\infty_{\net}(X)\right) \leq w^\top \Sigma^{\infty}_{\net}(X) w
	= \tfrac{1}{d}(s-2)^2 m_{d,s}B_s^2(X)\epsilon_x^2.
\end{align*}
Meanwhile, the function \(g_k(u;w)\) in \eqref{E:g_k(u;w)} satisfies
\begin{align*}
	g_k(u;w) = \alpha\left(x_{1k}\sigma_s^{(3)}(u^\top x_1)|u|^2 + 2\sigma_s^{(2)}(u^\top x_1)u_k\right)
	+ \beta\left(x_{2k}\sigma_s^{(3)}(u^\top x_2)|u|^2 + 2\sigma_s^{(2)}(u^\top x_2)u_k\right).
\end{align*}
Thus, we have 
\begin{align*}
	|g_k(u;w)|^2 \leq & 2\left|\alpha x_{1k}\sigma_s^{(3)}(u^\top x_1)|u|^2 
		+ \beta x_{2k}\sigma_s^{(3)}(u^\top x_2)|u|^2\right|^2\\
		&+ 8 \left| \alpha\sigma_s^{(2)}(u^\top x_1)u_k + \beta\sigma_s^{(2)}(u^\top x_2)u_k \right|^2 \\
	 =& 2 |u|^{2(s-1)}\left|\alpha x_{1k}\sigma_s^{(3)}(\ol{u}^\top x_1) 
		+ \beta x_{2k}\sigma_s^{(3)}(\ol{u}^\top x_2)\right|^2 \\
	&+ 8 u_k^2 |u|^{2(s-2)} \left| \alpha\sigma_s^{(2)}(\ol{u}^\top x_1) + \beta\sigma_s^{(2)}(\ol{u}^\top x_2) \right|^2.
\end{align*}
For the first term, by the choice of \(\alpha\) and \(\beta\), we have 
\begin{align*}
	&\left|\alpha x_{1k}\sigma_s^{(3)}(\ol{u}^\top x_1) + \beta x_{2k}\sigma_s^{(3)}(\ol{u}^\top x_2)\right|\\
	 =& (s-2)\left|\alpha (s)_2|x_1|^{s-2} \ol{x}_{1k}\sigma_{s-3}(\ol{u}^\top \ol{x}_1) + 
	\beta (s)_2|x_2|^{s-2} \ol{x}_{2k}\sigma_{s-3}(\ol{u}^\top \ol{x}_2)\right| \\
 \leq& (s-2)B_s(X)
	\left|\ol{x}_{1k}\sigma_{s-3}(\ol{u}^\top \ol{x}_1) - \ol{x}_{2k}\sigma_{s-3}(\ol{u}^\top \ol{x}_2)\right|,
\end{align*}
with \(\ol{x}_{ik} = x_{ik}/|x_{i}|\) for \(i=1,2\).
Furthermore, there holds
\begin{align*}
	&\left|\ol{x}_{1k}\sigma_{s-3}(\ol{u}^\top \ol{x}_1) - \ol{x}_{2k}\sigma_{s-3}(\ol{u}^\top \ol{x}_2)\right|\\
	 \leq &\left|\ol{x}_{1k}\sigma_{s-3}(\ol{u}^\top \ol{x}_1) - \ol{x}_{2k}\sigma_{s-3}(\ol{u}^\top \ol{x}_1) \right|
	+ \left|\ol{x}_{2k}\sigma_{s-3}(\ol{u}^\top \ol{x}_1) - \ol{x}_{2k}\sigma_{s-3}(\ol{u}^\top \ol{x}_2) \right|\\
 \leq &\left|\ol{x}_{1k} -  \ol{x}_{2k} \right| 
	+ |\ol{x}_{2k}|\left|\sigma_{s-3}(\ol{u}^\top \ol{x}_1) - \sigma_{s-3}(\ol{u}^\top \ol{x}_2) \right|.
\end{align*}
Therefore, we have 
\begin{align*}
	&\sum_{k=1}^{d}\left|\ol{x}_{1k}\sigma_{s-3}(\ol{u}^\top \ol{x}_1) - \ol{x}_{2k}\sigma_{s-3}(\ol{u}^\top \ol{x}_2)\right|^2\\
	 \leq &\sum_{k=1}^{d} 2\left|\ol{x}_{1k} -  \ol{x}_{2k} \right|^2 + 
	\sum_{k=1}^{d} 2|\ol{x}_{2k}|^2\left|\sigma_{s-3}(\ol{u}^\top \ol{x}_1) - \sigma_{s-3}(\ol{u}^\top \ol{x}_2) \right|^2\\
	=& 2\left|\ol{x}_{1} -  \ol{x}_{2} \right|^2 + 
		2\left|\sigma_{s-3}(\ol{u}^\top \ol{x}_1) - \sigma_{s-3}(\ol{u}^\top \ol{x}_2) \right|^2.
\end{align*}
Note that \(|\sigma_{s}(x) - \sigma_{s}(y)|\leq (s+1)|x-y|\) for \(s\geq 0\) and \(|x|,|y|\leq 1\).
Then, we have 
\begin{align*}
	&\sum_{k=1}^{d} \left|\alpha x_{1k}\sigma_s^{(3)}(\ol{u}^\top x_1) 
		+ \beta x_{2k}\sigma_s^{(3)}(\ol{u}^\top x_2)\right|^2\\ 
	 \leq &2(s-2)^2 B^2_{s}(X) ( |\ol{x}_{1} -  \ol{x}_{2} |^2 + |\sigma_{s-3}(\ol{u}^\top \ol{x}_1) - \sigma_{s-3}(\ol{u}^\top \ol{x}_2)|^2) \\
	 \leq &2(s-1)(s-2)^2 B^2_{s}(X) \left|\ol{x}_{1} -  \ol{x}_{2} \right|^2.
\end{align*}
Then, we have 
\begin{align*}
	\sum_{k=1}^{d}|g_k(u;w)|^2 
	 \leq& 4|u|^{2(s-1)} (s-1)(s-2)^2 B^2_{s}(X) \left|\ol{x}_{1} -  \ol{x}_{2} \right|^2 \\
     &  + 8|u|^{2(s-1)} \left| \alpha\sigma_s^{(2)}(\ol{u}^\top x_1) 
    + \beta\sigma_s^{(2)}(\ol{u}^\top x_2) \right|^2.
\end{align*}
Taking expectation gives
\begin{align*}
	&\E_{u\sim\NN(0,I)}\left[ \sum_{k=1}^{d}|g_k(u;w)|^2 \right] \\
	 \leq& 4m_{d,s-1}(s-1)(s-2)^2 B^2_{s}(X)\epsilon_x^2  + 8m_{d,s-1}\E_{u\sim \mathrm{U}(\mathbb{S}^{d-1})} 
	\left[ \left| \alpha\sigma_s^{(2)}(\ol{u}^\top x_1) + \beta\sigma_s^{(2)}(\ol{u}^\top x_2) \right|^2\right] \\
	 \leq& 4m_{d,s-1}(s-1)(s-2)^2 B^2_{s}(X)\epsilon_x^2 + 8m_{d,s-1} \frac{1}{d}(s-2)^2 B^2_{s}(X)\epsilon_x^2 
	 \leq 4m_{d,s-1}s(s-2)^2 B^2_{s}(X)\epsilon_x^2.
\end{align*}
Therefore, we have 
\begin{align*}
	\lambda_{\min}\left(\Theta^\infty_{\net}(X)\right)
	& \leq w^\top \Sigma^{\infty}_{\net}(X) w + \E_{u\sim\NN(0,I)}\left[ \sum_{k=1}^{d}|g_k(u;w)|^2 \right] \\
	& \leq \left(\frac{1}{d}m_{d,s} + 4sm_{d,s-1}\right) (s-2)^2 B^2_{s}(X)\epsilon_x^2 \\
	& = \left( \frac{1}{d}\cdot 2^s\frac{\Gamma(s+\frac{d}{2})}{\Gamma(\frac{d}{2})} + 
		4s\cdot 2^{s-1}\frac{\Gamma(s+\frac{d}{2})}{(s-1+\frac{d}{2})\Gamma(\frac{d}{2})} \right) (s-2)^2 B^2_{s}(X)\epsilon_x^2 \\
	& \leq 2^{s+2}(s-2)^2\frac{\Gamma(s+\frac{d}{2})}{\Gamma(\frac{d}{2})} B^2_{s}(X)\epsilon_x^2.
\end{align*}
Similarly, if \(\epsilon_y = |\ol{y}_1-\ol{y}_2| = \min\left\{ |\ol{y}_i-\ol{y}_j|: i\neq j\in[N_2] \right\}\), 
then let the test vector \(w=[0,\ldots,0,\alpha,\beta,0,\ldots,0]^\top\) with \(\alpha^2+\beta^2=1\) and 
\(\alpha,\beta\) located at the positions corresponding to \(y_1,y_2\) respectively. 
We can choose \(\alpha,\beta\) so that \(\alpha |y_1|^{s} = -\beta |y_2|^{s}\). 
Then, we have  
\begin{align*}
	w^\top \Sigma^{\infty}_{\net}(X) w & = m_{d,s} 
	\E_{u\sim \mathrm{U}(\mathbb{S}^{d-1})} [ | \alpha
	\sigma_s(u^\top y_1) + \beta \sigma_s(u^\top y_2) |^2]
	\leq \tfrac{1}{d}s^2 m_{d,s}B_s^2(X)\epsilon_y^2.
\end{align*}
Similarly, we have 
$g_k(u;w) = \alpha y_{1k}\sigma'_s(u^\top y_1) + \beta y_{2k}\sigma'_s(u^\top y_2).$ 
Thus, 
\begin{align*}
	\sum_{k=1}^{d} |g_k(u;w)|^2
	& \leq 2|u|^{2(s-1)} s^2 B^2_{s}(X) \left( \left|\ol{y}_{1} -  \ol{y}_{2} \right|^2 + 
		\left|\sigma_{s-1}(\ol{u}^\top \ol{y}_1) - \sigma_{s-1}(\ol{u}^\top \ol{y}_2) \right|^2 \right) \\
	& \leq 2|u|^{2(s-1)} s^3 B^2_{s}(X) \left|\ol{y}_{1} -  \ol{y}_{2} \right|^2.
\end{align*}
Therefore, we have 
\begin{align*}
	\lambda_{\min}\left(\Theta^\infty_{\net}(X)\right)
	& \leq w^\top \Sigma^{\infty}_{\net}(X) w + \E_{u\sim\NN(0,I)}\left[ \sum_{k=1}^{d}|g_k(u;w)|^2 \right] \\
	& \leq \left(\frac{1}{d}m_{d,s} + 2sm_{d,s-1}\right) s^2 B^2_{s}(X)\epsilon_y^2 
	\leq 2^{s+1}s^2\frac{\Gamma(s+\frac{d}{2})}{\Gamma(\frac{d}{2})} B^2_{s}(X)\epsilon_y^2.
\end{align*}

\subsection{Lower bound on the largest eigenvalue}
For \eqref{E:Inner-DNTK}, we have for any \(w\in \mathbb{S}^{N-1}\),  
\begin{align*}
	\lambda_{\max}\left(\Theta^\infty_{\net}(X)\right) & \geq w^\top \Theta^\infty_{\net}(X)w 
	\geq w^\top \Sigma^\infty_{\net}(X)w 
	= m_{d,s} \E_{u\sim \mathrm{U}(\mathbb{S}^{d-1})}\left[|g(u;w)|^2\right].
\end{align*}
Fix \(x_1\in \arg\max\left\{ (s)_2^2|x_p|^{2(s-2)}, |y_q|^{2s} \right\}\) and let \(w=(1,0,\ldots,0)\). Then, 
\begin{align*}
	w^\top \Sigma^{\infty}_{\net}(X) w 
	& = m_{d,s} \E_{u\sim \mathrm{U}(\mathbb{S}^{d-1})} \left[ \left| \sigma_s^{(2)}(u^\top x_1) \right|^2\right] 
	  = m_{d,s} B_s^2(X)
		\E_{u\sim \mathrm{U}(\mathbb{S}^{d-1})}\left[ \left| \sigma_{s-2}(u^\top \ol{x}_1) \right|^2\right] \\
	& = 2^s\frac{\Gamma(s+\frac{d}{2})}{\Gamma(\frac{d}{2})}
		B_s^2(X)\frac{1}{2\sqrt{\pi}}\frac{\Gamma(\frac{d}{2})\Gamma(s-\frac{3}{2})}{\Gamma(s-2+\frac{d}{2})}
	  = \frac{2^{s-1}}{\sqrt{\pi}} \frac{\Gamma(s+\frac{d}{2})\Gamma(s-\frac{3}{2})}{\Gamma(s-2+\frac{d}{2})} B_s^2(X).
\end{align*}
Similarly, if \(y_1\in \arg\max\left\{ (s)_2^2|x_p|^{2(s-2)}, |y_q|^{2s} \right\}\), let \(w=(0,\ldots,0,1,0,\ldots,0)\)
with one lying at the \(N_1+1\)-th position. Then, we get 
\begin{align*}
	&w^\top \Sigma^\infty_{\net}(X) w  = m_{d,s} B_s^2(X) 
	\E_{u\sim \mathrm{U}(\mathbb{S}^{d-1})}\left[ \left| \sigma_{s}(u^\top \ol{y}_1) \right|^2\right] \\
	 = &2^s\frac{\Gamma(s+\frac{d}{2})}{\Gamma(\frac{d}{2})} B_s^2(X) 
		\frac{1}{2\sqrt{\pi}}\frac{\Gamma(\frac{d}{2})\Gamma(s+\frac{1}{2})}{\Gamma(s+\frac{d}{2})} = \frac{2^{s-1}}{\sqrt{\pi}} \Gamma(s+\tfrac{1}{2}) B_s^2(X).
\end{align*}
Note that \(\frac{\Gamma(s+\frac{d}{2})\Gamma(s-\frac{3}{2})}{\Gamma(s-2+\frac{d}{2})}\geq \Gamma\left(s+\frac{1}{2}\right)\) as \(s\geq 3, d\geq 2\).
Therefore, we have 
\[ \lambda_{\max}\left(\Theta^\infty_{\net}(X)\right) \geq \lambda_{\max}\left(\Sigma^\infty_{\net}(X)\right)
   \geq \frac{2^{s-1}}{\sqrt{\pi}} \Gamma(s+\tfrac{1}{2}) B_s^2(X). \]

\subsection{Upper bound on the largest eigenvalue}
We first estimate \(\lambda_{\max}\left(\Sigma^\infty_{\net}(X)\right)\), corresponding to the first term in 
the expression of \(\Theta^\infty_{\net}(X)\) \eqref{E:Inner-DNTK}.
Note that 
\[ \lambda_{\max}\left(\Sigma^\infty_{\net}(X)\right) 
	= \sup_{w\in \mathbb{S}^{N-1}} w^\top \Sigma^\infty_{\net}(X) w
	= m_{d,s} \sup_{w\in \mathbb{S}^{N-1}} \E_{u\sim \mathrm{U}(\mathbb{S}^{d-1})}\left[|g(u;w)|^2\right]. \]
By the identity \eqref{E:g(u;w)} and the Cauchy-Schwarz inequality, we have 
\begin{align*}
	|g(u;w)|^2 &\leq \sum_{p=1}^{N_1} (s)_2^2|x_p|^{2(s-2)}\sigma_{s-2}(u^\top \ol{x}_p)^2 
	+ \sum_{q=1}^{N_2} |y_q|^{2s}\sigma_s(u^\top \ol{y}_q)^2 \\
	&\leq B_s^2(X) \bigg(\sum_{p=1}^{N_1} \sigma_{s-2}(u^\top \ol{x}_p)^2 + 
	\sum_{q=1}^{N_2} \sigma_s(u^\top \ol{y}_q)^2\bigg).
\end{align*}
By equation \eqref{E:E[sigma(ux_1)^2]} and \(\frac{\Gamma\left(s+\frac{1}{2}\right)}{\Gamma(s+\frac{d}{2})}\leq \frac{\Gamma(s-\frac{3}{2})}{\Gamma(s-2+\frac{d}{2})}\), we have 
\begin{align*}
	\E_{u\sim \mathrm{U}(\mathbb{S}^{d-1})}\left[|g(u;w)|^2\right] \leq &B_s^2(X) 
	\left( \frac{N_1}{2\sqrt{\pi}}\frac{\Gamma(\frac{d}{2})\Gamma(s-\frac{3}{2})}{\Gamma(s-2+\frac{d}{2})} 
	 + \frac{N_2}{2\sqrt{\pi}}\frac{\Gamma(\frac{d}{2})\Gamma(s+\frac{1}{2})}{\Gamma(s+\frac{d}{2})} \right)\\
	 \leq &\frac{N B_s^2(X)}{2\sqrt{\pi}}\frac{\Gamma(\frac{d}{2})\Gamma(s-\frac{3}{2})}{\Gamma(s-2+\frac{d}{2})}.
\end{align*}
Next, we estimate the second term in the expression of \(\Theta^\infty_{\net}(X)\). 
From \eqref{E:g_k(u;w)}, we have
\begin{align*}
	g_k(u;w) = \sum_{p=1}^{N_1} w^{I}_p \left( x_{pk} |u|^2 
	\sigma_s^{(3)}\left(u^\top x_p \right) + 2u_k\sigma_s^{(2)}\left(u^\top x_p \right) \right)
	+ \sum_{q=1}^{N_2} w^{B}_q \left(y_{qk}\sigma_s^{(1)}\left(u^\top y_q\right)\right).
\end{align*}
By the Cauchy-Schwarz inequality, we have for any \(w\in\mathbb{S}^{N-1}\), 
\begin{align*}
	|g_k(u;w)|_2^2 \leq \sum_{p=1}^{N_1} 2\left(x_{pk}^2 |u|^4 \sigma_s^{(3)}\left(u^\top x_p \right)^2
	+ 4u_k^2 \sigma_s^{(2)}\left(u^\top x_p \right)^2  \right)
	+ \sum_{q=1}^{N_2} y_{qk}^2\sigma_s^{(1)}\left(u^\top y_q\right)^2.
\end{align*}
By summing the inequality over \(k\in[d]\), we get for any \(w\in \mathbb{S}^{d-1}\),
\begin{align*}
	 \sum_{k=1}^{d} \left| g_k(u;w) \right|^2 
	\leq& \sum_{p=1}^{N_1} 2|x_p|^2 |u|^4 \sigma_s^{(3)}\left(u^\top x_p \right)^2 + 
	8|u|^2 \sigma_s^{(2)}\left(u^\top x_p \right)^2
	+ \sum_{q=1}^{N_2} |y_{q}|^2\sigma_s^{(1)}\left(u^\top y_q\right)^2 \\
	 = &\sum_{p=1}^{N_1} (s)_2^2|x_p|^{2(s-2)} |u|^{2(s-1)} \left(2(s-2)^{2}\sigma_{s-3}(\ol{u}^\top \ol{x}_p )^2 
	+8\sigma_{s-2}(\ol{u}^\top \ol{x}_p )^2 \right)\\
	&+ \sum_{q=1}^{N_2} |y_{q}|^{2s}|u|^{2(s-1)} s^2\sigma_{s-1}(\ol{u}^\top \ol{y}_q)^2.
\end{align*}
Taking expectation with respect to \(u\) gives
\begin{align*}
	\E_{u\sim \NN(0,I_d)} \left[ \sum_{k=1}^{d} | g_k(u;w) |^2 \right]
	& \leq m_{d,s-1} B_s^2(X) \E_{u\sim \mathrm{U}(\mathbb{S}^{d-1})} \left[ h(u) \right], 
\end{align*}
with the function \(h(u)\in L^2(\mathbb{S}^{d-1})\) given by
\[ h(u) = \sum_{p=1}^{N_1} \left(2(s-2)^{2}\sigma_{s-3}(\ol{u}^\top \ol{x}_p )^2 
	+8\sigma_{s-2}(u^\top \ol{x}_p )^2 \right)
	+ \sum_{q=1}^{N_2} s^2\sigma_{s-1}(u^\top \ol{y}_q)^2.  \]
By equation \eqref{E:E[sigma(ux_1)^2]} and the inequality \(\frac{\Gamma(s-\frac{1}{2})}{\Gamma(s-1+\frac{d}{2})}\leq \frac{\Gamma(s-\frac{3}{2})}{\Gamma(s-2+\frac{d}{2})}\leq \frac{\Gamma(s-\frac{5}{2})}{\Gamma(s-3+\frac{d}{2})}\), we deduce 
\begin{align*}
	& \E_{u\sim \mathrm{U}(\mathbb{S}^{d-1})} \left[ h(u) \right] \\
	& = \sum_{p=1}^{N_1} \left(2(s-2)^{2}\frac{1}{2\sqrt{\pi}}\frac{\Gamma(\frac{d}{2})\Gamma(s-\frac{5}{2})}{\Gamma(s-3+\frac{d}{2})}
	+8\frac{1}{2\sqrt{\pi}}\frac{\Gamma(\frac{d}{2})\Gamma(s-\frac{3}{2})}{\Gamma(s-2+\frac{d}{2})} \right)
	+ \sum_{q=1}^{N_2} s^2\frac{1}{2\sqrt{\pi}}\frac{\Gamma(\frac{d}{2})\Gamma(s-\frac{1}{2})}{\Gamma(s-1+\frac{d}{2})} \\
	& \leq N_1\left( \frac{(s-2)^2}{\sqrt{\pi}} 
	+ \frac{4}{\sqrt{\pi}} \right)\frac{\Gamma(\frac{d}{2})\Gamma(s-\frac{5}{2})}{\Gamma(s-3+\frac{d}{2})}
	+ N_2 \frac{s^2}{2\sqrt{\pi}}\frac{\Gamma(\frac{d}{2})\Gamma(s-\frac{5}{2})}{\Gamma(s-3+\frac{d}{2})} \\ 
    & \leq \frac{s(s-1)}{\sqrt{\pi}}\frac{\Gamma(\frac{d}{2})\Gamma(s-\frac{5}{2})}{\Gamma(s-3+\frac{d}{2})} N.
\end{align*}
Therefore, the largest eigenvalue $\lambda_{\max}(\Theta^\infty_f(X))$ of \(\Theta^\infty_{\net}(X)\) is upper bounded by
\begin{align*}
	\lambda_{\max}\left(\Theta^\infty_{\net}(X)\right) 
	& \leq \sup_{w\in \mathbb{S}^{N-1}} w^\top \Sigma^{\infty}_{\net}(X) w + 
	\sup_{w\in \mathbb{S}^{N-1}} \E_{u} \left[ \sum_{k=1}^{d} \left\| g_k(u;w) \right\|^2 \right] \\
	& \leq  m_{d,s} \frac{N B_s^2(X)}{2\sqrt{\pi}}\frac{\Gamma(\frac{d}{2})\Gamma(s-\frac{3}{2})}{\Gamma(s-2+\frac{d}{2})}
	  + m_{d,s-1} B_s^2(X) \frac{s(s-1)}{\sqrt{\pi}}\frac{\Gamma(\frac{d}{2})\Gamma(s-\frac{5}{2})}{\Gamma(s-3+\frac{d}{2})} N \\
	& = \left(\frac{2^{s-1}}{\sqrt{\pi}}\frac{\Gamma(s+\frac{d}{2})\Gamma(s-\frac{3}{2})}{\Gamma(s-2+\frac{d}{2})}
		+ \frac{2^{s-1}s(s-1)}{\sqrt{\pi}}\frac{\Gamma(s-1+\frac{d}{2})\Gamma(s-\frac{5}{2})}{\Gamma(s-3+\frac{d}{2})}\right)N B_s^2(X) \\
	& \leq \frac{2^{s-1}s^2}{\sqrt{\pi}}\frac{\Gamma(s+\frac{d}{2})\Gamma(s-\frac{5}{2})}{\Gamma(s-2+\frac{d}{2})} N B_s^2(X), 
\end{align*}
where we have used \(\Gamma(s-\frac{3}{2})=(s-\frac{5}{2})\Gamma(s-\frac{5}{2})\) and \(\frac{\Gamma(s-1+\frac{d}{2})}{\Gamma(s-3+\frac{d}{2})}\leq \frac{\Gamma(s+\frac{d}{2})}{\Gamma(s-2+\frac{d}{2})}\) for \(s\geq 3, d\geq 2\).

% \begin{comment}
% \section{Discussion} \label{S:Discussion}

% In this work, we theoretically study the spectrum of the DNTK, and 
% provide the explicit lower and upper bound for the extreme eigenvalues of the DNTK matrix, given an arbitrary set of samplings points, for the shallow RePU network without bias, for the Poisson equation with the Dirichlet boundary condition. 
% The method can be also applied to other symmetric PDEs like Helmholtz and 
% biharmonic equations, and to Neumann boundary condition for the spherical domain whose boundary is a centered sphere.
% %The following are some inadequacy in this paper that needs to further study in the furture. 
% First, the eigenvalue bounds estimates of the infinite DNTK matrices are restricted to shallow network without bias term and symmetric PDEs. In constrast, people often use DNNs with bias term.
% %due to the benefits in various aspects,  
% e.g., increasing the expressivity of networks or the positivity of DNTK \cite{JinWu:2026DNTK}. 
% Moreover, instead of the RePU activation,  smooth activations like softplus, sigmoid and hyperbolic tangent functions are quite popular in machine learning tasks. Meanwhile, non-symmetric and nonlinear PDEs are commonly encountered in scientific problems, thus it is also of great interest to develop the analysis in these cases.
% \end{comment}

% \appendix

%\section*{Acknowledgments} We would like to acknowledge the assistance 

\bibliographystyle{abbrv}
\bibliography{reference}

\end{document}